\documentclass{article}
\usepackage[english]{babel}
\usepackage{amsmath,amssymb,xcolor}
\usepackage{amsthm}
\usepackage{graphicx}
\usepackage{hyperref}
\usepackage{cases}

\theoremstyle{plain}
\newtheorem{thm}{Theorem}
\newtheorem{lem}[thm]{Lemma}

\newtheorem{rem}{Remark}
\newtheorem{cor}{Corollary}

\theoremstyle{definition}
\newtheorem{defn}{Definition}[section]

\theoremstyle{remark}

\newcommand{\norm}[1]{\Arrowvert #1 \Arrowvert}
\newcommand{\R}{\mathbb{T}}
\newcommand{\Rd}{\mathbb{T}^d}

\newcommand{\PP}{\mathbb{P}}
\newcommand{\N}{\mathbb{N}}

\begin{document}
	\title{Quantitative particle approximation of nonlinear stochastic  Fokker-Planck equations with singular kernel}
	\author{Josu\'e Knorst, Christian Olivera and Alexandre B. de Souza}
	\date{}

	\maketitle

	\begin{abstract}
		We derive quantitative estimates for large stochastic systems of interacting particles perturbed by both idiosyncratic and environmental noises, as well as singular kernels. We prove that the (mollified) empirical process converges  to the solution of the nonlinear stochastic Fokker-Planck equation. The proof is based on  It\^o's formula for  $H_{q}^{1}$-valued process, commutator estimates,  and some estimations for the 
		regularization of the empirical measure. Moreover, we show that the aforementioned equation admits a unique strong solution
in the probabilistic sense. The approach applies to repulsive and attractive kernels.

	\end{abstract}

	\vspace{0.3cm} \noindent {\bf MSC2010 subject classification:} 49N90,  60H30, 60K35.

	\section{Introduction}

	The aim of this article is to derive a quantitative convergence result  for the following  
	nonlinear stochastic Fokker-Planck equation

	\begin{equation}\label{SPDE_Ito}
		\mathrm{d}\rho_t = \frac{1}{2} \sum_{i, j=1}^d \partial_{i j} \left(\rho_t \sum_{k=1}^{d}\left(\nu_t^{ik}\nu_t^{jk} +\sigma_t^{i k} \sigma_t^{j k}\right)\right)\, dt  - \nabla \cdot  \big(\rho_t F(\cdot\,,  K\ast \rho_t)\big)\, d  t - \nabla \rho_t \cdot \sigma_t\, d  B_t 
	\end{equation}

	\noindent from a stochastic moderately interacting particle system. The equation (\ref{SPDE_Ito})  is considered for arbitrary dimensions $d \geq 1$.
	We present a method that allows us to derive a fairly general class of such systems as limit dynamics of interacting stochastic many-particle systems. This procedure has been methodically employed in the literature, from the seminal papers  \cite{Oelschlager84, Oelschlager85,  Oelschlager87},  (see also  \cite{JourdainMeleard} and \cite{Meleard}) to the recent works \cite{FlandoliOliveraSimon} and \cite{Pisa}. Concerning motivations for the model with transport noise, let us mention model reduction, see \cite{Maj}, in addition to other motivations like  \cite{Flanlect} and  \cite{Holm}.
	Our results  cover several classical models such as the stochastic  $2d$ Navier-Stokes equation, which in vorticity form can be written as in \eqref{SPDE_Ito} with the Biot-Savart kernel,  the stochastic Burgers 
	equation  and the parabolic-elliptic Keller-Segel PDE in any dimension $d\geq 1$, which models the phenomenon of chemotaxis.

	The derivation of macroscopic models from interacting particle systems is a fascinating and active research topic in mathematics.
	The basic idea is that, as the number of particles increases to infinity, the macroscopic model can effectively describe the universal properties of particles. In this article, we derive the stochastic PDE  \label{SPDE_Ito} from the stochastic 
	moderately interacting particle system on 
	$\mathbb{T}^d$ given by 
\begin{align}\label{particles}
d  X_t^{i,N} &= F \left(X_t^{i,N}, \frac{1}{N}\sum_{k=1}^{N} \left(K \ast V^{N}\right)\left(X_{t}^{i,N} -X_{t}^{k,N}\right)\right) \, dt +  \nu_t\left(X_t^{i,N}\right)\,d W_t^{i}\nonumber\\
&+ \sigma_t\left(X^{i,N}_t\right)\, dB_t  \quad 
\end{align}
	
	\noindent where $W_{t}^{i}$ and $B_{t}$ are  independent standard $\mathbb{T}^d$-valued Brownian motions, defined on a filtered probability space $(\Omega,\mathcal{F},(\mathcal{F}_t)_{t\geq0},\mathbb{P})$, the interaction kernel $V^{N} $ depends on the number of particles  $N \in \mathbb{N} $ via the moderate interaction parameter $\beta$. Systems of the form (\ref{particles}) are commonly employed in the field of mean-field games, see \cite{Carmona}, \cite{Carmona2}, \cite{Live} and \cite{Lac}.

	The microscopic \emph{empirical process} of this $N$-particle system, which is a probability measure on the ambient space $\mathbb{T}^d$,  is given as usual by 
	\begin{equation} S_{t}^{N}:=\frac{1}{N} \sum_{i=1}^{N}\delta_{X_{t}^{i,N}}, \qquad t\geqslant 0, \label{eq:mutN}\end{equation} where $\delta_{a}$ is the delta Dirac measure concentrated at $a \in \R^d$. Then, $(S_{t}^{N})_{t\geqslant 0}$ is a  measure--valued process associated to the $\R^d$--valued processes $\{t\mapsto X_{t}^{i,N}\}_{i=1,...,N}$. 
	
	The first contribution of  this paper is to investigate the large $N$ limit of  the dynamical process
	$(S_{t}^{N})_{t\geqslant 0}$ in the common noise setting with singular kernels. For that purpose, we introduce the \emph{mollified empirical measure} \[\rho_t^{N}:= V^{N} \ast S_t^{N} = \int_{\R^d} V^N(\cdot-y) \, S_t^N(\mathrm{d} y),\] 
	which is more regular than $S_t^N$, where  $V^N = N^\beta V(N^{\frac{\beta}{d}} \cdot)$, $\beta \in (0,1)$.
	We quantify the distance between the empirical measure (and the regularized version) and the unique solution 
	to  the SPDE  (\ref{SPDE_Ito}).  The proof  is  based  on     Itô's formula for $H_{q}^{1}$-valued processes, commutator  estimates,  and  stochastic  analysis techniques.\\  
	
	The second  contribution is to show  the well-posedness of the nonlinear stochastic Fokker–Planck equation (\ref{SPDE_Ito}). The main challenge in proving existence and uniqueness results for
	(\ref{SPDE_Ito}) is the nonlinear term $ K\ast \rho_t$ since this prevents us from directly applying known results in the existing literature on SPDEs, such as those found in \cite{kry}. We  take advantage of the specific SPDE structure to employ a fixed point argument, e.g., the recent work  \cite{Hi}.

	\subsection{Related literatures}
	In contrast to interacting particle system, which are driven by idiosyncratic noise, see \cite{Bres, Carri,Catti,Font,FournierJourdain,FHM, Tardy, Gui,JabinWang,Nguyen,  Pisa, Serfaty, Toma}, the literature on common noise and singular kernels  remains limited. For systems with uniformly Lipschitz interaction coefficients, \cite{Cogui} established conditional propagation of chaos. The entropy method has recently been explored for systems with common noise, as shown in \cite{Shao} for the Navier–Stokes equations, \cite{Chen2} for the Hegselmann–Krause model, and \cite{Niko} for mean-field systems with bounded kernels. 

	In the following, we will focus only on the results of moderately interacting systems:  in  \cite{Oelschlager84}, Oelschläger introduced and studied moderately interacting particle systems, which are used to obtain a local nonlinear partial differential equation. The idea is to take a radially symmetric function $V$ and then consider the interaction potential $V^N = N^\beta V(N^{\frac{\beta}{d}} \cdot)$,    with   $\beta \in (0,1)$, which converges to the Dirac delta at $0$ in the sense of distributions. The article \cite{Oelschlager84}  is part of a series of works by the author on this subject (see also \cite{Oelschlager85,Oelschlager87}). The convergence results were improved  in \cite{JourdainMeleard} and \cite{Meleard}. In the recent paper \cite{FlandoliLeimbachOlivera}
	it was  developed a semigroup approach which enables them to show uniform
	convergence of mollified empirical measures,  see \cite{FlandoliLeocata,Leocata, FlandoliOliveraSimon,Live, Simon, Guo} for additional applications of this method.  Recently, in \cite{Pisa} and \cite{ORT}, the authors developed new quantitative estimates for a class of moderately interacting particle systems using a semigroup approach. This technique was further improved in \cite{Hao}, \cite{Knorst}, and \cite{Simon2}. About more advances in  moderate particle systems,  see, for instance, \cite{Ansgar}, \cite{Chen}, \cite{correa},  \cite{Ansgar2}, and \cite{Steve}.
	
	The problem   of deriving   models represented by SPDE with  transport noise  and singular kernels
	from microscopic models of  stochastic interacting particles
	has not been highly investigated in the literature, just with the contributions in \cite{Chen2}, \cite{Niko} and   \cite{Shao}. As we will see  in more detail in Section \ref{examples}, our result may be applied to the following well-known models:
	the stochastic  $2d$ Navier-Stokes equation,  the stochastic Burgers 
	equation  and the  stochastic parabolic-elliptic Keller-Segel PDE in any dimension $d\geq 1$.

	\subsection{Notations}
	For $d \geq 1 $, we denote by $ C^{k}\left(\mathbb{T}^d\right) $ the space of $ k $-times differentiable functions on the torus $ \mathbb{T}^d \doteq \left[-\frac{1}{2}, \frac{1}{2}\right]^d $, where $ k \in \mathbb{N} \cup \{\infty\}$. The set of density functions defined on $ \mathbb{T}^d $ is denoted by $ \mathcal{P}\left(\mathbb{T}^d\right) $. The corresponding spaces on $\mathbb{R}^d$ are defined by  $ C^{k}\left(\mathbb{R}^d\right)$ and  $\mathcal{P}\left(\mathbb{R}^d\right)$, respectively.
	
	For a measurable space $(X,\mathcal{M},\mu)$ and $f $ a measurable function in it (which we denote by $f \in \mathcal{M}$) we write 
	$$\left\langle\mu,f\right\rangle \doteq \int_{X}fd\mu$$
	for the duality pair involving the measure $\mu$ and the function $f$ in this space.
	Also, for  $a \in [1,\infty)$  the Lebesgue's space is given by
	$$L^a=L^a(X)=\left\{f \in \mathcal{M} \mid\|f\|_{a}\doteq \left(\int_{X}|f|^a d\mu\right)^{\frac{1}{a}}<\infty\right\}
	$$
	and if $a=\infty$ 
	$$L^{\infty}=L^{\infty}(X)=\left\{f \in \mathcal{M} \mid\|f\|_{\infty}\doteq \text{ess} \sup_{x \in X}|f(x)|<\infty\right\}.
	$$
	In some contexts we will write $\|\cdot\|_{a} =\|\cdot\|_{L^a(X)}$.
	
	Given $(U,\|\cdot\|_U)$ a Banach space, we denote by $U^*$ its dual and the associated duality pair by $(\cdot,\cdot)_{U,U^*}$.
	With this notation, for $T >0$ and denoting $\mathcal{L}$ the set of Bochner measurable functions, the Bochner's space is denoted by 
	\begin{equation*}
		\resizebox{\hsize}{!}{%
			$L^aU = L^a([0,T];U) = \left\{f:[0,T] \to U, f \in \mathcal{L}\mid \|f\|_{L^aU} \doteq \left(\int_0^T\|f(t)\|_U^a \ dt \right)^{\frac{1}{a}}< \infty\right\},$%
		}
	\end{equation*}
	for $a \in [1,\infty)$ and if $a=\infty$ by
	\begin{equation*}
		\resizebox{\hsize}{!}{%
			$L^{\infty}U=L^{\infty}([0,T];U) = \left\{f:[0,T] \to U, f \in \mathcal{L}\mid \|f\|_{T,U} \doteq \text{ess} \sup_{t\in[0,T]}\|f(t)\|_U < \infty\right\}.                 $%
		}
	\end{equation*}

	For the space of tempered distributions on $\mathbb{T}^d$, we write $\mathcal{S}'$.  So, for $q > 1, n \in \mathbb{R}$,  the Bessel potential space reads
	$$H_{q}^{n}=H_{q}^{n}\left(\mathbb{T}^{d}\right)=\left\{f \in \mathcal{S}' \mid\|f\|_{n,q}\doteq\left\|(I-\Delta)^{\frac{n}{2}} f\right\|_{q}<\infty\right\}.
	$$
	
	\noindent For $\gamma \in (0,1]$ the H\"older's space on $\mathbb{T}^d$ is given by $$C^{\gamma}=C^{\gamma}\left(\mathbb{T}^d\right)=\left\{f:\mathbb{T}^d \to \mathbb{R} \mid\|f\|_{\gamma}\doteq \|f\|_{\infty} + \sup_{x\neq y\in \mathbb{T}^d}\frac{|f(x)-f(y)|}{|x-y|^{\gamma}}<\infty\right\}.$$
	
	Let a filtered probability space $(\Omega,\mathcal{F},(\mathcal{F}_t)_{t\geq0},\mathbb{P})$, a Banach space $(U,\|\cdot\|_U)$, a real number $ q \in [1,\infty]$, and a stopping time $0 <\tau \leq T$. We  denote by  $\mathcal{X}$ the set of $U$-valued, $\left(\mathcal{F}_t\right)_{t \in[0, T]}$-adapted and continuous processes $X =\left\{X_{s}\right\}_{s \in[0, \tau]}$. We define 
	$$S_{\mathcal{F}}^{q}([0, \tau] ; U) = \left\{X \in \mathcal{X} \mid \Big\|\|X\|_{\tau,U}\Big\|_{L^q(\Omega)} < \infty\right\}$$
		
		Also, let $\mathcal{Y}$ denote the set of  $U$-valued predictable processes $Y=\left\{Y_{s}\right\}_{s \in[0, \tau]}$. Then, for $q \in [1,\infty)$ we define
		$$L_{\mathcal{F}}^{q}([0,\tau];U) = \left\{Y \in \mathcal{Y}\mid\Big\|\|Y\|_{L^qU}\Big\|_{L^q(\Omega)} < \infty\right\}.$$
		
			These are Banach spaces endowed with their respective norms.
			
			\subsection{Assumptions}
			
			\begin{enumerate}
				\item[$(\mathbf{A}^V)$] The mollifier $V \in C^2\left(\mathbb{R}^d\right) \cap \mathcal{P}\left(\mathbb{R}^d\right)$ is such that $\text{supp} \, V \subset \left\{x \in \mathbb{R}^d; |x|< \frac{1}{2}\right\}$. 
				
				\item[$(\mathbf{A}^F)$] The vector field $F: \mathbb{T}^d \times \mathbb{R} \to \mathbb{R}^d$ is a bounded Lipschitz continuous function, i.e., there exists $L>0$ such that
				\begin{align*}
					&|F(x,u)| \le L, \quad \forall x \in \mathbb{T}^d , \  u\in \mathbb{R}.\\
					&|F(x,u) - F(y,v)| \le L ( |x-y| + |u-v|), \quad \forall x,y \in \mathbb{T}^d , \  u,v \in \mathbb{R}.
				\end{align*}
				
				\item[$(\mathbf{A}^{I})$] The vector field $F: \mathbb{T}^d \times \mathbb{R} \to \mathbb{R}$ is given by  $F(x,u) = u$.
				
				\item[$(\mathbf{A}^K)$] There exists $q \ge 2$ such that $q > d$, and $\gamma \in(0,1]$ and $C_K>0$ such that for any $f \in  L^q (\mathbb{T}^d)$, it holds \[ \|K \ast f\|_{\gamma} \le C_K\|f\|_{q}. \]
				\item[$(\mathbf{A}^{c})$]
				The diffusion coefficients $\nu$ and $\sigma: [0,T] \times \mathbb{T}^d \to \mathbb{R}^{d \times d}$ are measurable, and for every $t \in [0,T]$ satisfy
				
				$(\mathbf{A}^{c}_i)$ There exists $M>0$ such that
				\[ \nu_t^{jk}(\cdot),\sigma_t^{jk}(\cdot) \in \left(C^2\left(\mathbb{T}^d\right),\| \cdot \|_{C^2}\right), \quad \text{for } j,k=1,...,d. \]  
				\[ \text{and } \quad \sum_{j,k=1}^d   \Big( \|\nu_t^{jk}(\cdot)\|_{C^2} + \|\sigma_t^{jk}(\cdot)\|_{C^2} \Big) \le M.    \]
				\[ \hspace{-90px} (\mathbf{A}^{c}_{ii}) \hspace{15px}  \sum_{j=1}^d \partial_j \sigma_t^{jk}(x) = 0, \, \quad \text{for } k=1,...,d; \, x \in \Rd.   \]
				$(\mathbf{A}^{c}_{iii})$ There exist $\Lambda,\lambda>0$ such that, for all $x \in \mathbb{T}^d$, $ \xi=(\xi^i) \in \mathbb{R}^d$,
				\[ \Lambda |\xi|^2\geq \sum_{i,j,k=1}^d \nu_t^{ik}(x) \nu_t^{jk}(x)\, \xi^i\xi^j \ge \lambda \, |\xi|^2. \]
			\end{enumerate}
			
			\subsection{Main theorems.}
			
			Regarding the limiting equation \eqref{SPDE_Ito}, we will prove the following existence and uniqueness theorem, based on Krylov's $L^q$-theory of SPDEs, see \cite{kry}.
			
			\begin{thm}\label{Teo_Krylov} Assume $(\mathbf{A}^c)$ and one of the two regimes:
				\begin{enumerate}
					\item $(\mathbf{A}^F)$ and $(\mathbf{A}^K)$.
					\item $(\mathbf{A}^I)$ and $(\mathbf{A}^K)$.
									\end{enumerate}
				Let $0 \le \rho_0 \in L^1 \cap H^{1-\frac{2}{q}}_q\left(\mathbb{T}^d\right)$, with $\norm{\rho_0}_{1}=1$. For each $\kappa>0$, there exists a $T>0$ depending on $\kappa,\lambda, q, C_K, L, M$ and $d$ such that if $\norm{\rho_0}_{q} = \kappa$,  the SPDE (\ref{SPDE_Ito}) admits a unique nonnegative solution in
				\begin{align*}
					\mathbb{M} \doteq L_{\mathcal{F}^B}^q\left([0, T]; H_q^{1}\left(\mathbb{T}^d\right)\right) \cap S_{\mathcal{F}^B}^{\infty}\left([0, T] ; L^1 \cap L^q\left(\mathbb{T}^d\right)\right).
				\end{align*}
			\end{thm}
			The proof will be presented in Appendix \ref{apC}.

			\begin{thm} \label{first main}
				Assume $(\mathbf{A}^V)$, $(\mathbf{A}^c)$, $(\mathbf{A}^F)$,  and $(\mathbf{A}^K)$, let $T_{max}$ be the maximal existence time for (\ref{SPDE_Ito}) and fix $T \in \left(0,T_{max}\right)$. In addition, let the dynamics of the particle system be given by (\ref{particles}) with $\nu=Id$, $\sigma$ independent of the space  and  for any $m \ge 1$,
				\begin{align*}
					\sup_{N \in \mathbb{N}} \Big \| \left\| \rho^N_0 \right\|_{q} \Big \|_{{L^m(\Omega)}} < \infty
				\end{align*}
				with $ d \geq 1$, $q>d$ given by $(\mathbf{A}^K)$ and  $\beta \in \left(0,\frac{1}{2\big[1 + \frac{1}{d} - \frac{1}{q}\big]}\right) $.
				Then
				
				\begin{align*}
					\Big \| \left\|\rho - \rho^N \right\|_{T,q} \Big \|_{L^m(\Omega)}  &\le C \Big \| \left\|\rho_0 - \rho^N_0 \right\|_{q} \Big \|_{L^m(\Omega)} + C N^{-\varkappa} \, \nonumber 
				\end{align*}
				where
				\begin{align*}
					\varkappa = \min \left(\frac{\beta}{d}\gamma,\frac{1}{2}-\beta\Big(1 + \frac{1 }{d}-\frac{1}{q}\Big) \right)
				\end{align*}
				\noindent and $\rho$ is a  solution of SPDE \eqref{SPDE_Ito} with initial condition $\rho_0$.
			\end{thm}
			
			\begin{thm}\label{teo_d1}
				Assume $(\mathbf{A}^V)$, $(\mathbf{A}^c)$, $(\mathbf{A}^F)$, $K = \delta_0$, let $T_{max}$ be the maximal existence time for (\ref{SPDE_Ito}) and fix $T \in (0,T_{max})$. In addition, let the dynamics of the particle system be given by (\ref{Ndensity}) with $\nu=Id$, $\sigma$ independent of the space and for any $m\geq1$, 
				\begin{align*}
					\sup_{N \in \mathbb{N}} \Big \| \left\| \rho^N_0 \right\|_{2} \Big \|_{L^m(\Omega)} < \infty
				\end{align*}
				with $d=1$ and $ \beta \in \left(0,\frac{1}{3}\right) $.
				Then
				\begin{align*}
					\Big \| \left\|\rho - \rho^N \right\|_{T,2} \Big \|_{L^m(\Omega)}  &\le C \Big \| \left\|\rho_0 - \rho^N_0 \right\|_{2} \Big \|_{L^m(\Omega)} \nonumber + C N^{-\varkappa} \, \nonumber 
				\end{align*}
				where
				\begin{align*}
					\varkappa = \min \left(\frac{\beta}{2}, \frac{1}{2}-\frac{3}{2}\beta \right)
				\end{align*}
				\noindent and $\rho$ is a  solution of SPDE \eqref{SPDE_Ito} with initial condition $\rho_0$.
			\end{thm}

			In view of the previous result, we obtain a rate of convergence for the genuine empirical measure, which can be interpreted as a propagation of chaos for the marginals of the empirical measure of the particle system. Following \cite[Section 8.3]{BogachevII}, let us introduce the Kantorovich-Rubinstein metric which reads, for any two probability measures $\mu$ and $\nu$ on $\R^d$,
			\begin{equation}\label{eq:defWasserstein}
				\|\mu - \nu \|_{0} = \sup \left\{ \int_{\R^d} \phi \, d(\mu-\nu) \, ; ~ \phi \text{ Lipschitz  with } \|\phi\|_{L^\infty(\R^d)}\leq 1 \text{ and } \|\phi\|_{\text{Lip}} \leq 1 \right\} .
			\end{equation}

			\begin{cor} \label{cor:rateEmpMeas}
				Let the same assumptions as in Theorem~\ref{first main}  or  Theorem~\ref{teo_d1} be in place. 
				Then for any $\varepsilon \in (0,\varkappa)$, there exists $C>0$ such that, for any $N\in \mathbb{N}$,
				\[
				\bigg\| \sup_{t\in[0,T]}  \left\|S_{t}^N - \rho_{t} \right\|_{0} \bigg\|_{L^m(\Omega)}  \leq 
				C \left(  \bigg\|  \left\|S^N_0-\rho_0\right\|_{0} \bigg\|_{L^m(\Omega)} + \, N^{-\varkappa + \varepsilon} \right) .
				\]
			\end{cor}

			\begin{proof}
				Let $t\in (0,T_{max})$. Let us observe first that 
				there exists $C>0$ such that for any Lipschitz continuous function $\phi$ on $\R^d$, one has
				\begin{align}\label{eq:rateunmun}
					\left|\left\langle \rho^N_{t},\phi\right\rangle - \left\langle S^N_{t},\phi\right\rangle \right| \leq \frac{C \|\phi\|_{\text{Lip}}}{N^{\frac{\beta}{d}}} \quad a.s.
				\end{align}
				Indeed, 
				\begin{align*}
					\left|\left\langle S_{t}^{N}, \phi\right\rangle - \left\langle \rho^N_t, \phi \right\rangle \right| &=\left| \left\langle S_{t}^{N}, \left(\phi-\phi\ast V^{N}\right)\right\rangle \right|\\
					&\leq \left\langle S_{t}^{N}, \int_{\R^{d}} V(y)~ |\phi(\cdot)-  \phi\left(\cdot- \frac{y}{N^{\beta}}\right) |   dy \right\rangle \\
					&\leq \frac{C \|\phi\|_{\text{Lip}}}{N^{\frac{\beta}{d}}}.
				\end{align*}

				Recalling the definition \eqref{eq:defWasserstein} of the Kantorovich-Rubinstein distance, it comes
				\begin{align*}
					\left\| \sup_{t\in[0,T]} \left\| S_{t}^N - \rho_{t} \right\|_{0}  \right\|_{m} &\leq \left\| \sup_{t\in[0,T]} \left\| S_{t}^N - \rho^N_{t}\right\|_{0} \right\|_{m} +  \left\| \sup_{t\in[0,T]} \sup_{ \|\phi\|_{L^\infty}\leq 1} \left\langle \rho_{t}^N - \rho_{t},\phi \right\rangle \right\|_{m} \\
					&\leq \left\| \sup_{t\in[0,T]} \left\| S_{t}^N - \rho^N_{t}\right\|_{0} \right\|_{m} +  \left\| \sup_{t\in[0,T]}  \left\| \rho_{t}^N - \rho_{t}\right\|_{L^1(\R^d)} \right\|_{m},
				\end{align*}
				where we are using the following notation for the norm in $L^m(\Omega)$: $\| \cdot \|_{L^m(\Omega)} = \| \cdot \|_m$.
				Now applying inequality \eqref{eq:rateunmun} to the first term on the right-hand side of the above inequality, and Theorem \ref{first main} (or Theorem \ref{teo_d1})  to the second term, we obtain the inequality of Corollary \ref{cor:rateEmpMeas}.
				
			\end{proof}

			\subsection{Application for some  stochastic PDE } \label{examples}
			
			Let  $A>0$ and $f_A: \mathbb{R}  \rightarrow \mathbb{R} $ be a $C^2(\mathbb{R})$ bounded function such that
			\begin{enumerate}
				\item $f_A(x)=x$, for $x\in [-A,A]$,
				\item $f_A(x)= A$, for $x> A+1$ and $f_A(x)= -A$, for $x< -(A+1)$,
				\item $\|f_{A}'\|_\infty \leq 1$ and $\|f_{A}''\|_\infty <\infty$.
			\end{enumerate}

			As a consequence, $\|f_A\|_\infty\leq A+1$.
			Now $F_{A}$ is given by
			\begin{equation}\label{eq:defF0}
				F_{A}  : (x_{1},\dots, x_{d})^{\top} \mapsto \left(f_{A}(x_{1}) ,\dots, f_{A}(x_{d})\right)^{\top} .
			\end{equation}

			When $A=\infty$ then  $ F_{A}=Id$. Now, we consider  the stochastic PDE 
			
			\begin{equation}\label{Stratonich}
				\mathrm{d}\rho_t = \frac{1}{2} \Delta  \rho_t  dt  - \nabla \cdot  \big(\rho_t (  K\ast \rho_t)\big)\, d  t - \nabla \rho_t \cdot \sigma_t\, \circ d  B_t 
			\end{equation}
			
			\noindent where the integral is in the Stratonovich sense, or  equivalently in It\^o  formulation
			\begin{equation}\label{Ito2}
				\mathrm{d}\rho_t = \frac{1}{2}\left( \Delta \rho_t + \sum_{i, j=1}^d \partial_{i j} \rho_t  \sum_{k=1}^{d} \sigma_t^{i k} \sigma_t^{j k}\right)
				dt  - \nabla \cdot  \big(\rho_t (K\ast \rho_t)\big)\, d  t - \nabla \rho_t \cdot \sigma_t\, d  B_t.
			\end{equation}
			
			We consider the associated stochastic  moderate particle  system

			\begin{equation}\label{particles2}
				d  X_t^{i,N} = F_{A} \left( K \ast \rho_t^{N} \left(X_t^{i,N}\right)\right) \, dt +  \,d W_t^{i} + \sigma_t\, 
				dB_t  \quad 
			\end{equation}

			\noindent and also the particle system without cutoff:
			\begin{equation}\label{particles3}
				d\widetilde{X}_{t}^{i,N}= \frac{1}{N}\sum_{k=1}^{N} (K \ast V^{N})(\widetilde{X}_{t}^{i,N} -\widetilde{X}_{t}^{k,N})\; dt +\; dW_{t}^{i} +  \sigma_t 
				\,dB_t, \,\,\, t\leq T,\,\,\,  1\leq i \leq N. 
			\end{equation}
			
			For fixed $N\in \N^*$, the drift term is smooth and bounded. Hence, this particle system has a unique strong solution. We denote by $\widetilde{S}_{t}^N$ its empirical distribution, and define 
			$\widetilde{\rho}^N = V^N\ast \widetilde{S}^N$.   We denote  $A_{T}:=C_{K} \Big\|\left\|\rho\right\|_{T,q} \Big\|_{L^{\infty}(\Omega)}$  where $C_{K}$ is given by   $(\mathbf{A}^K)$ and 
			$\rho$ is the unique solution of  equation (\ref{Ito2}). Then   $\rho$ is also the unique solution of the equation

			\begin{equation}\label{Ito3}
				\mathrm{d}\rho_t = \frac{1}{2}\left( \Delta \rho_t + \sum_{i, j=1}^d \partial_{i j} \rho_t  \sum_{k=1}^{d} \sigma_t^{i k} \sigma_t^{j k}\right)
				dt  - \nabla \cdot  \big(\rho_t F_{A}(K\ast \rho_t)\big)\, d  t - \nabla \rho_t \cdot \sigma_t\, d  B_t 
			\end{equation}

			The proof of the following result  is based on 
			Corollary 2.3 in \cite{Pisa}.  
			
			\begin{cor}\label{import}	
				Assume  $(\mathbf{A}^c)$, $(\mathbf{A}^V)$,  $(\mathbf{A}^I)$ and $(\mathbf{A}^K)$  and let $\rho $ the unique solution of the SPDE 
				(\ref{Stratonich})  in $ L_{\mathcal{F}^B}^q\left([0, T]; H_q^{1}\left(\mathbb{T}^d\right)\right) \cap S_{\mathcal{F}^B}^{\infty}\left([0, T] ; L^1 \cap L^q\left(\mathbb{T}^d\right)\right)$, given by Theorem \ref{Teo_Krylov}.  In addition, let the dynamics of the particle system be given 
				by (\ref{particles3})   and  
				for any $m \ge 1$,
				
				\begin{align*}
					\sup_{N \in \mathbb{N}} \Big \| \left\|\widetilde{\rho}^N_0 \right\|_{q} \Big \|_{{L^m(\Omega)}} < \infty
				\end{align*}
				with $ d \geq 1$, $q>d$ given by $\left(\mathbf{A}^K\right)$ and  $\beta \in \left(0,\frac{1}{2\big[1 + \frac{1}{d} - \frac{1}{q}\big]}\right) $. Then 
				we have 
				
				\begin{align*}
					\PP \left(  \left\|\widetilde{\rho}_{t}^N - \rho_{t} \right\|_{T,q} \geq \eta  \right) \leq \frac{C}{\eta^m} \left( \Big\|  \left\|\widetilde{\rho}^N_0- \rho_0\right\|_{T,q} \Big\|_{L^m(\Omega)}  + C N^{-\varkappa}\right)^m .
				\end{align*}

			\end{cor}

			\begin{proof}
				Let us introduce some notations to distinguish the particle systems with and without cutoff. 
				We fix $A>A_{T}$ to be precisely chosen later.
				We consider the particles $\left(X^{i,N}\right)_{1\leq i\leq N}$ and $\left(\widetilde{X}^{i,N}\right)_{1\leq i\leq N}$ defined respectively by \eqref{particles2} and \eqref{particles3} with the same initial conditions and driven by the same family of independent $\R^d$-valued  Brownian motions $(W^i)_{i\in \N^*}$ and $B_{t}$. 
				For any $N\in \N^*$ and $A>0$, define
				\begin{align*}
					\Omega_{N,A} = \Bigg\{ \sup_{\substack{t\in [0,T] \\ i\in \{1,\dots,N\}}} \frac{1}{N}\bigg| \sum_{k=1}^{N} \left(K \ast V^{N}\right)\left(X_{t}^{i,N} - X_{t}^{k,N}\right)\bigg| \leq A\Bigg\} 
				\end{align*}
				and observe that on $\Omega_{N,A}$, we have $X^{i,N}_{t} = \widetilde{X}^{i,N}_{t}$ for all $t\in [0,T]$ and all $i\in \{1,\dots,N\}$.  Since $\frac{1}{N} \sum_{k=1}^{N} \left(K \ast V^{N}\right)\left(X_{t}^{i,N} - X_{t}^{k,N}\right) = K \ast \rho^N_{t}(X^{i,N}_{t})$, we also get that on $\Omega_{N,A}$, $\rho^N_{t} = \widetilde{\rho}^N_{t}$ for all $t\in [0,T]$ and all $i\in \{1,\dots,N\}$. Hence
				\begin{align*}
					\PP \left(  \left\|\widetilde{\rho}^N - \rho \right\|_{T,q} \geq \eta \right) &= \PP \left(\Omega_{N,A}^c \cap \left\{ \left\|\widetilde{\rho}^N - \rho \right\|_{T,q} \geq \eta \right\}\right) \\
					&\hspace{1cm}+\PP \left(\Omega_{N,A} \cap \left\{ \left\|\widetilde{\rho}^N - \rho \right\|_{T,q} \geq \eta \right\}\right) \\
					&\leq \PP \left(\Omega_{N,A}^c \right) +\PP \left( \left\| \rho^N - \rho \right\|_{T,q} \geq \eta \right) .
				\end{align*}

				Now by hypothesis  we have that $\left|K \ast \rho^N_{t}\left(X^{i,N}_{t}\right)\right| \leq C_{K} \|\rho^N\|_{T,q}$. Thus we get that for $A= C_{K} \left(\eta + \Big\|\left\|\rho\right\|_{T,q}\Big\|_{L^{\infty}(\Omega)}\right)$,
				\begin{align*}
					\PP \left(\Omega_{N,A}^c \right) &\leq \PP \left(\left\|\rho^N\right\|_{T,q} >\frac{A}{C_{K}}\right)\\
					&\leq \PP \left(\left\|\rho\right\|_{T,q} + \left\|\rho^N - \rho\right\|_{T,q} >\frac{A}{C_{K}}\right)\\
					&\leq \PP \left(\Big\|\left\|\rho\right\|_{T,q} \Big\|_{L^{\infty}(\Omega)} + \left\|\rho^N - \rho\right\|_{T,q} >\frac{A}{C_{K}}\right)\\
					&\leq \PP \left( \left\|\rho^N - \rho\right\|_{T,q} >\eta\right).
				\end{align*}
				
				Hence 
				\begin{align*}
					\PP \left( \left\|\widetilde{\rho}^N - \rho \right\|_{T,q} \geq \eta \right) \leq 2 \PP \left( \left\|\rho^{N} - \rho\right\|_{T,q} >\eta\right).
				\end{align*}
				Now using Markov's inequality and Theorem \ref{first main} with $F(x,u)=F_A(u)$, we obtain the desired result.
			\end{proof}

			\paragraph{Biot-Savart kernel and the $2d$ Navier-Stokes equation.}
			
			By considering the vorticity field $\xi_t$ associated to the stochastic incompressible two-dimensional Navier-Stokes solution $\rho_t$, 
			one gets equation \eqref{Ito2}   with the Biot-Savart kernel $K(x)=\frac{i}{2\pi} \sum_{k\in \mathbb{Z}^2, k\neq 0  } \exp(ik\cdot x) \, \frac{k^{\bot}}{k^{2}} $.
			The original Navier-Stokes solution is then recovered thanks to the formula $\rho_{t} = K\ast \xi_{t}$. 
			The kernel $K\in L^{p}(\mathbb{T}^2)$ with $p<2$ and $\|\nabla K \ast f \|_{q}\leq C \|f \|_{q} $ with $1<q<\infty$. Then
			$\|  K \ast f\|_{1, q} \leq C \| f\|_{q}  $ with $q>2$, thus  by Sobolev embedding $\|  K \ast f\|_{C^{1-\frac{2}{q}}(\mathbb{T}^2)} \leq C \| f\|_{q}$. This model  arises naturally in fluid mechanics; see, for instance, \cite{Brez} and \cite{FlandoliLuongo}. This case is covered by Corollary \ref{import}, also we can apply theorem \ref{first main} with  cutoff  $A> C_{K} \Big\|\left\|\rho\right\|_{T,q} \Big\|_{L^{\infty}(\Omega)}$ in the SPDE 
(\ref{Ito3})  where $C_{K}$ is given by   $(\mathbf{A}^K)$.

			\paragraph{Parabolic-elliptic Keller-Segel models in any dimension.} 
			An important and tricky example covered by this paper is the parabolic-elliptic Keller-Segel PDE, which takes the form \eqref{Ito2} where  the kernel $K$
			is the periodization of the function defined on $\mathbb{T}^d$ by $K_{0}(x)= -\chi_d \frac{x}{|x|^d}$ for some $\chi_d>0$. The $K$ verifies for all 
			$  f\in L^q(\mathbb{T}^d),  \  \| K\ast f  \|_{C^{1-\frac{d}{q} }} \leq C\, \|f\|_{L^q(\mathbb{T}^d)} $, see Lemma 41 in \cite{ORT}. 
			This case is also covered by Corollary \ref{import}, also we can apply theorem \ref{first main} with  cutoff  $A> C_{K} \Big\|\left\|\rho\right\|_{T,q} \Big\|_{L^{\infty}(\Omega)}$ in the SPDE 
(\ref{Ito3})  where $C_{K}$ is given by   $(\mathbf{A}^K)$.

			\paragraph{Burgers  Equation.} When $d = 1$, $K = \delta_0$ the equation (\ref{Stratonich})-(\ref{Ito2})
			is the  well-known stochastic scalar Burgers equation.  If the initial 
			condition $\rho_{0}\in H_{2}^{2}$ then by lemma 4.10 in \cite{Alonso} the solution
			verifies the maximum  principle, that is, $\| \rho \|_{\infty}\leq C_{T}\|\rho_{0}\|_{\infty}$. Taking $A> C_{T}\|\rho_{0}\|$  the solution of the equation (\ref{Ito2}) is also  solution of the equation 
			(\ref{Ito3}). The Theorem  \ref{teo_d1} covers this case.

			\smallskip
			\subsection{Heuristic deduction}
			
			For the ease of reading, we may write $X^i_t$ in place of $X^{i,N}_t$. \\ The starting $X_0^{i,N}$ are i.i.d. with common law $\mu_0$ with density $\rho_0$.
			
			Let $x \in \mathbb{T}^d$. Applying It{\^o}'s formula with the smooth function $V^N (x - \cdot)$ for each $i \in \{1, \ldots, N \}$ we have
				\begin{align*}
				V^N \left(x - X_t^{i}\right) & = \, V^N \left(x - X_0^{i}\right) \\ 
				&- \sum_{j=1}^d\int_0^t\partial_jV^N \left(x - X_s^{i}\right)      F_j \left(X_s^{i}, K\ast \rho_s^N \left(X_s^{i}\right)\right) \, d  s\\
				& + \frac{1}{2}\sum_{j,k=1}^d \int_0^t\partial_{jk}V^N \left(x - X_s^{i}\right) 
				\left(\sigma \sigma^{\top}\right)_s^{jk}(X_s^i) \, d  s\\
				& - \sum_{j,k=1}^d \int_0^t\partial_{j}V^N \left(x - X_s^{i}\right)  
				\sigma_s^{jk}\left(X^i_s\right) \, d  B_s^k\\
				& \,  \, - \sum_{j,k=1}^d\int_0^t\partial_jV^N \left(x - X_s^{i}\right) \nu_s^{jk}\left(X_s^i\right) \, d  W_s^{i,k} \\
				&+ \frac{1}{2} \sum_{j,k=1}^d \int_0^t \partial_{jk} V^N \left(x - X_s^{i}\right)\left(\nu\nu^{\top}\right)_s^{jk}\left(X_s^i\right) \, d s .
			\end{align*}
			with $\left(\sigma \sigma^{\top}\right)^{jk}_s \doteq  \sum_{l = 1}^{d} \sigma^{jl}_s\sigma^{kl}_s$ and $\left(\nu \nu^{\top}\right)^{jk}_s \doteq  \sum_{l = 1}^{d} \nu^{jl}_s\nu^{kl}_s$. Then since $\rho^N = V^N \ast S^N$ we have
			\begin{align}
				\rho_t^N(x) & = \, \rho^N_0(x) \nonumber \\
				&- \sum_{j=1}^d\int_{0}^{t}\left\langle S_{s}^{N}, \partial_jV^{N}(x-\cdot)F_j\left(\cdot, K\ast \rho^N_s(\cdot)\right)\right\rangle d s \nonumber\\
				& + \frac{1}{2}\sum_{j,k=1}^d\int_{0}^{t}\left\langle S_{s}^{N}, \partial_{jk}V^{N}(x-\cdot) \left(\sigma \sigma^{\top}\right)_s^{jk}(\cdot) \right\rangle d s \nonumber\\
				& -\sum_{j,k=1}^d \int_0^t\left\langle S_s^N, \partial_{j}V^N (x - \cdot) 
				\sigma_s^{jk}(\cdot) \right \rangle  d  B_s^k\nonumber\\
				& \,  \, - \frac{1}{N}\sum_{i=1}^N\sum_{j,k=1}^d\int_0^t\partial_{j}V^N \left(x - X^i_s\right) \nu_s^{jk}\left(X^i_s\right) \, dW_s^{i,k} \nonumber\\
				& \,  \, +\frac{1}{2}\sum_{j,k=1}^d\int_{0}^{t}\left\langle S_{s}^{N}, \partial_{jk}V^{N}(x-\cdot) \left(\nu \nu^{\top}\right)_s^{jk}(\cdot) \right\rangle d s . \label{Ndensity}
			\end{align}
			If  $\sigma$ and $\nu$ do not depend on the spatial variable we obtain 
			\begin{align}
				\rho_t^N(x) & = \, \rho^N_0(x) \nonumber \\
				&- \sum_{j=1}^d\int_{0}^{t}\left\langle S_{s}^{N}, \partial_jV^{N}(x-\cdot)F_j\left(\cdot, K\ast \rho^N_s(\cdot)\right)\right\rangle d s \nonumber\\
				& + \frac{1}{2}\sum_{j,k=1}^d\int_{0}^{t} \partial_{jk}\rho^N_s \left(\sigma \sigma^{\top}\right)_s^{jk}  d s \nonumber\\
				& -\sum_{j,k=1}^d \int_0^t \partial_{j}\rho^N_s 
				\sigma_s^{jk} d  B_s^k\quad \quad \quad \big(=\sigma_s^{\top} \nabla \rho_s^N \,dB_s\big)\nonumber\\
				& \,  \, - \frac{1}{N} \sum_{i=1}^N \sum_{j,k=1}^d \int_0^t \partial_{j}V^N \left(x - X^i_s \right) \nu_s^{jk} \, dW_s^{i,k} \quad \,\,\big( = \nu^{\top}_s\nabla V^N(x-X^i_s) \, dW^i_s\big)  \nonumber\\
				& \,  \, +\frac{1}{2}\sum_{j,k=1}^d\int_{0}^{t}\partial_{jk}\rho^N_s \left(\nu \nu^{\top}\right)_s^{jk}d s . \label{Ndensity_aditivo}
			\end{align}
			
			Under these conditions, we may conjecture that
			a limiting measure-valued deterministic process  $\rho_t$ exists and its
			evolution equation (in weak form) is given by

			\begin{align}
				\mathrm{d}\rho_t &= \frac{1}{2} \sum_{i, j=1}^d \partial_{i j} \left(\rho_t \sum_{k=1}^{d}(\nu_t^{ik}\nu_t^{jk} +\sigma_t^{i k} \sigma_t^{j k})\right)\, dt \nonumber\\ &\quad - \nabla \cdot  \big(\rho_t F(\cdot\,,  K\ast \rho_t)\big)\, d  t - \nabla \rho_t \cdot \sigma_t\, d  B_t .
			\end{align}
			
			\subsection{Definition of solution }
			
			\begin{defn} A family of random functions $\left\{\rho_t(\omega): t \ge 0, \omega \in \Omega\right\}$ lying in $S_{\mathcal{F} ^B}^{\infty}\left([0, T] ; L^1 \cap L^q\left(\mathbb{T}^d\right)\right)$ is a solution to (\ref{SPDE_Ito}) if $\rho_t$ satisfies the following stochastic integral equation for all $\phi \in C^2\left(\mathbb{T}^d\right)$,
				
\begin{align}
\left\langle \rho_t, \phi\right\rangle &= \left\langle \rho_0, \phi\right\rangle+ \int_0^t\left\langle \rho_s F(\cdot\,, K \ast \rho_s), \nabla \phi \right\rangle ds \nonumber \\
& \quad +\frac{1}{2} \int_0^t\left\langle \rho_s, \sum_{i,j=1}^d \partial_{ij}\phi \,   \sum_{k=1}^d (\nu_s^{ik} \, \nu_s^{jk}+ \sigma_s^{ik} \, \sigma_s^{jk}) \right\rangle \, ds \nonumber \\
&\quad +\int_0^t\left\langle \rho_s, \sum_{i=1}^d \partial_i \phi\,  \sum_{k=1}^{d} \sigma_s^{ik} \right\rangle \, d B_s^k.   \label{weak_sol}
\end{align}
\end{defn}
			
			
			\medskip

			\bigskip
			\begin{rem}\label{continhas}
				Let $q\geq2$. The following will be helpful throughout the text.
				\vspace{.4cm}
				\begin{itemize}
					\item[(i)]  $\nabla (| \rho (x) |^q)  = \, q | \rho (x) |^{q - 2} \rho (x) \nabla \rho(x)$. \\
					\item[(ii)] $\nabla (q | \rho (x) |^{q - 2}  \rho (x))= \, q (q - 1) | \rho (x) |^{q -	2}\nabla \rho (x)$. \\\end{itemize}
			\end{rem}

			\medskip
			
			\section{Proof of the main theorems}
			\subsection{Proof of Theorem \ref{first main}}
			\begin{proof}

	Applying to $(\rho - \rho^N)$ the Itô's formula for the $L^q$-norm of a $H_{q}^{1}$-valued process in \cite{kry} we have 
	
	\begin{align*}    &\|\rho_t - \rho_t^N\|_q^q = \, \|\rho_0 - \rho_0^N\|_q^q\\
		& \quad - \sum_{j,k=1}^d \int_0^t \int_{\R^d}q |\rho_s - \rho_s^N|^{q-2}(\rho_s - \rho_s^N)\partial_{j}(\rho_s - \rho_s^N)\sigma_s^{jk} \,dx \, dB_s^k\\
		& \quad  - \frac{1}{N}\sum_{i=1}^N\int_0^t\int_{\R^d}q |\rho_s - \rho_s^N|^{q-2}(\rho_s -\rho_s^N)\left(\nabla V^N\right) (x - X_s^{i}) \, dx \, dW_s^{i} \\
		& \quad +q(q-1)\int_{0}^{t}\int_{\mathbb{T}^d}|\rho_s - \rho_s^N|^{q-2}\nabla (\rho_s -\rho_s^N)\, \Big[ \rho_s F(x,K\ast \rho_s) \\
		& \quad \hspace{130px} -\left\langle S_{s}^{N},V^{N}(x-\cdot)\cdot F(\cdot, K\ast \rho_{s}^{N}(\cdot))\right\rangle \Big] dx\, d s\\
		& \quad + \frac{1}{2}\sum_{j,k=1}^d\int_{0}^{t} \int_{\mathbb{T}^d}q|\rho_s - \rho_s^N|^{q-2}(\rho_s - \rho_s^N)\partial_{jk}(\rho_s - \rho_s^N) \left(\sigma \sigma^T\right)_s^{jk} dx\, d s\\
		& \quad  - \frac{1}{2}  \int_0^t \int_{\R^d}q(q-1) |\rho_s -\rho_s^N|^{q-2}|\nabla (\rho_s -\rho_s^N)|^2 \, dx\, d s\\
		& \quad + \frac{1}{2} \int_{0}^{t}\int_{\mathbb{T}^{d}}  q(q-1)|\rho_s - \rho_{s}^{N}|^{q-2}\left|\sigma^T_s \nabla(\rho_s - \rho_s^N) \right|^2 \, dx \, ds\\
		& \quad+ \frac{1}{2 N^{2}} \sum_{i=1}^{N} \int_{0}^{t}\int_{\mathbb{T}^{d}}  q(q-1)\left|\rho_s -  \rho_{s}^{N}\right|^{q-2}\left| \left(\nabla V^N\right)(x-X_{s}^{i})\right|^{2} \, dx \, ds.
	\end{align*}
	
	\noindent Integration by parts gives us (recall that $\sigma$  does not depend on the spatial variable)
	\begin{align*}
		&\frac{1}{2}\sum_{j,k=1}^d\int_{0}^{t} \int_{\mathbb{T}^d}q|\rho_s - \rho_s^N|^{q-2}(\rho_s - \rho_s^N)\partial_{jk}(\rho_s - \rho_s^N) \left(\sigma \sigma^T\right)_s^{jk} \, dx \,d s \\ &= -\frac{1}{2} \int_{0}^{t}\int_{\mathbb{T}^{d}}  q(q-1)\left|\rho_s - \rho_{s}^{N}\right|^{q-2}\left|\sigma^T_s \nabla(\rho_s - \rho_s^N) \right|^2 \, dx \, ds, 
	\end{align*}
	hence we get 
	\begin{align}
		\left\| \rho_t - \rho^N_t\right\|_{q}^{q}&\le \left\| \rho_0 -\rho^N_0\right\|_{q}^{q}\nonumber+I_1 + I_2 \nonumber \\
		&-\frac{1}{2} \int_{0}^{t} \int_{\mathbb{T}^{d}} q(q-1)\left|\rho_s -  \rho^N_s\right|^{q-2}\left| \nabla (\rho_s - \rho^N_s)\right|^{2} \, dx \, ds \nonumber \\
		&-\sum_{j,k=1}^d\int_0^t\sigma^{jk}_s \underbrace{\int_{\R^d} \partial_j|\rho_s-\rho_s^N|^{q}\,dx }_{=0}d  B_s^k  \nonumber \\ 
		&-M_{t}^N,    \label{dec7second}
	\end{align}
	\noindent where 
	\begin{align*}
		I_1 \doteq q(q-1)\int_{0}^{t}\int_{\mathbb{T}^d}|\rho_s - \rho_s^N|^{q-2}\nabla (\rho_s& -\rho_s^N)\Big[\rho_sF\left(x, K\ast \rho_s\right) \\
		&-\left\langle S_{s}^{N}, V^{N}(x-\cdot)F\left(\cdot, K\ast \rho^N_s\right)\right\rangle \Big]\, dx\,ds,
	\end{align*}
	$$ I_2 \doteq \frac{1}{2 N^{2}} \sum_{i=1}^{N} \int_{0}^{t}\int_{\mathbb{T}^{d}}  q(q-1)\left|\rho_s - \rho_s^N\right|^{q-2}\left| \left(\nabla V^N\right)(x-X_{s}^{i})\right|^{2} \, dx \, ds,$$
	and
	$$M^N_{t} \doteq \frac{1}{N} \sum_{i=1}^{N}\int_{0}^{t} \int_{\mathbb{T}^d}q\left|\rho_s - \rho_s^N\right|^{q-2} (\rho_s- \rho_s^N) \,  ( \nabla V^{N})(x-X_{s}^{i}) \,dx \, dW_{s}^{i}.$$
	
	If $q >2$, by  Young inequality with $\frac{1}{\frac{q}{q-2}}+\frac{1}{\frac{q}{2}}=1$ we obtain
	\begin{align*}
		&I_2 = \frac{q(q-1)}{2 N} \sum_{i=1}^{N} \int_{0}^{t}\int_{\mathbb{T}^{d}} \left|\rho_s - \rho_s^N\right|^{q-2}\frac{1}{N}\left|\left(\nabla V^N\right)(x-X_{s}^{i})\right|^{2} \, dx \, ds \\
		&\hspace{-1px}\le \frac{q(q-1)}{2} \hspace{-3px}\int_{0}^{t}\left\| \rho_s - \rho_s^N\right\|_{q}^{q} ds+\frac{q(q-1)}{2N} \sum_{i=1}^{N} \int_{0}^{t} \hspace{-3px} \int_{\mathbb{T}^{d}}\frac{1}{N^\frac{q}{2}}\left|( \nabla V^{N})(x-X_{s}^{i})\right|^{q} d x \, d s 
	\end{align*}
	Now by (\ref{taxa derivada V}),
	\begin{align*}
		I_2 &\le C_2 \int_{0}^{t}\left\|\rho_s -  \rho^N_s\right\|_{q}^{q} d s+\frac{q(q-1)}{2}T N^{\left(-\frac{q}{2}+q\beta(1 + \frac{1 }{d})-\beta\right)}  \, \|\nabla V\|_{q}^q
	\end{align*}
	where $C_2 =\frac{q(q-1)}{2}$, if $q > 2$ and $C_2=0$, if $q=2$. 
	
	Now,  we  will estimate the first term  of  $I_1$. Since $F$ is Lipschitz,  $K$ satisfies $(\mathbf{A}^K)$,  and  $V \ge 0$  we have
	
	\begin{align*}
		\left| \left\langle S_s^{N}, \right. \right. & \left. \left. V^{N}(x - \cdot)\left(F\left(x, K\ast \rho_s^{N}(x)\right)-F(\cdot, K \ast \rho_s^{N}(\cdot))\right) \right \rangle \right| \\
		& \le\left\langle S_s^{N}, V^{N}(x-\cdot)\left|F\left(x, K \ast \rho_s^{N}(x)\right)-F\left(\cdot, K \ast \rho_s^{N}(\cdot)\right)\right|\right\rangle \\
		&\le L\left\langle S_s^{N}, V^{N}(x-\cdot)\left(|x - \cdot|+\left| K\ast (\rho_s^{N}(x)-\rho_s^{N}(\cdot))\right|\right) \right\rangle \\
		& \le L\left\langle S_s^{N}, V^{N}(x-\cdot)\left(|x - \cdot|+\left\|\rho_s^{N}\right\|_{q}|x-\cdot|^{\gamma}\right)\right\rangle \\
		& \le L\left(N^{-\frac{\beta}{d}}+\left\|\rho_s^{N}\right\|_{q} N^{-\frac{\beta}{d} \gamma }\right)\left|\rho_s^{N}(x)\right|,
	\end{align*}
	where in the last inequality we are using that $\text{supp} \, V \subset \left\{x \in \mathbb{R}^d\mid |x|< \frac{1}{2}\right\}$.  We add and subtract the terms $F(x,K\ast \rho_s^N)\rho_s$ and $F(x,K\ast \rho_s^N)\rho_s^N$ to deduce 
	\begin{align*} 
		&\left|F(x,K\ast \rho_s) \rho_s-\langle S_s^N, V^N(x-\cdot) F(\cdot, K\ast \rho_s^N)\rangle \right| \\
		& \le \left|F(x, K\ast \rho_s) \rho_s- F(x, K\ast \rho_s^N) \rho_s\right| \\
		&\quad+ \left|F(x, K\ast \rho_s^N)\rho_s- F(x, K\ast \rho_s^N) \rho^N_s\right| \\
		&\quad+ \langle S_s^N, V^N(x-\cdot)| F(x, K\ast \rho_s^N) - F(\cdot, K\ast \rho_s^N)|\rangle \\
		&\le L\left|K \ast (\rho_s - \rho_s^N)\right||\rho_s| \\
		&\quad+ L|\rho_s - \rho_s^N| \\
		&\quad+  L\left(N^{-\frac{\beta}{d}}+\left\|\rho_s^{N}\right\|_{q} N^{-\frac{\beta}{d} \gamma }\right)\left|\rho_s^{N}(x)\right| \\
		&\doteq (a+b+c).
	\end{align*}
	Now, using $\epsilon$-Young inequality for $\frac{1}{2} + \frac{1}{2} = 1$, we get
	\begin{align*}
		&|\nabla(\rho_s-\rho^N_s)|(a+b+c) \\
		&\le \epsilon |\nabla(\rho_s-\rho^N_s)|^2 + C_\epsilon (a+b+c)^2 \\ &\le \epsilon |\nabla(\rho_s-\rho^N_s)|^2 + C_\epsilon 3L^2\left(\left|K \ast (\rho_s - \rho_s^N)\right|^2|\rho_s|^2 +|\rho_s - \rho_s^N|^2 + C_N^2|\rho_s^N|^2 \right)
	\end{align*}
	where $C_N \doteq N^{-\frac{\beta}{d}}+\left\|\rho_s^{N}\right\|_{q} N^{-\frac{\beta}{d} \gamma}$. For the other constants, we will merge them into $C$, which may change from line to line. Thus we have 
	\begin{align*}
		|I_1|&\le  \epsilon \int_0^t \int_{\mathbb{T}^d} q (q - 1) | \rho_s - \rho_s^N |^{q-2} |\nabla (\rho_s - \rho_s^N)|^2 \, dx \, ds \\
		& +  C \int_0^t \int_{\mathbb{T}^d}  | \rho_s - \rho_s^N |^q \, dx \, ds. \\
		&+ C \int_0^t \int_{\mathbb{T}^d} | \rho_s - \rho_s^N |^{q-2}\big( \left|K \ast (\rho_s - \rho_s^N)\right|^2|\rho_s|^2+ C_N^2 |\rho_s^N|^2\big) \, dx \, ds. 
	\end{align*}
	We continue with the last term above. For $q > 2$, we use $1$-Young inequality for $\frac{1}{\frac{q}{q-2}} + \frac{1}{\frac{q}{2}} = 1$, 
	\begin{align*}
		&\hspace{-1px} \int_0^t \int_{\mathbb{T}^d} | \rho_s - \rho_s^N |^{q-2}\big( \left|K \ast (\rho_s - \rho_s^N)\right|^2|\rho_s|^2+ C_N^2 |\rho_s^N|^2\big) \, dx\, ds\\
		&\hspace{-1px}\le C \int_0^t \hspace{-2px}\int_{\mathbb{T}^d}| \rho_s - \rho_s^N |^q \, dx \, ds + C \int_0^t \hspace{-2px}\int_{\mathbb{T}^d} \|K \ast (\rho_s - \rho_s^N)\|_{\infty}^q \, |\rho_s|^q  + C_N^q  | \rho_s^N |^q \, dx \, ds \\
		&\hspace{-1px}\le C \int_0^t  \| \rho_s - \rho_s^N \|_q^q \, ds + C \|\rho\|_{T,q}^q  \int_0^t \|\rho_s - \rho_s^N\|_q^q  \, ds  + CT\,C_N^q  \| \rho^N \|_{T,q}^q   
	\end{align*}
	where in the last inequality we are using $(\mathbf{A}^K)$. Now, using that $\|\rho\|_{T,q}\le C$  a.s. due to Theorem \ref{Teo_Krylov}, we arrive at
	\begin{align}
		|I_1|&\le  \epsilon \int_0^t \int_{\mathbb{T}^d} q (q - 1) | \rho_s - \rho_s^N |^{q-2} \, |\nabla (\rho_s - \rho_s^N)|^2 \, ds\nonumber \\
		&\quad + C \int_0^t \| \rho_s - \rho_s^N \|_q^q \, ds + C \, C_N^q \|\rho^N \|_{T,q}^q.  \nonumber
	\end{align}
	
	Finally, we join the estimates with $\epsilon < \frac{1}{2}$ and continue from (\ref{dec7second}) to get
	\begin{align}
		\left\|\rho_t - \rho^N_t \right\|_q^{q}   &\le \left\|\rho_0 - \rho^N_0 \right\|_q^{q}+C\int_{0}^{t}\left\|\rho_s - \rho_s^N\right\|_{q}^{q} ds\nonumber\\
		&+CN^{\left(-\frac{q}{2}+q\beta(1 + \frac{1 }{d})-\beta\right)} \nonumber \\ 
		&+C\left(N^{-\frac{\beta}{d}}+\left\|\rho^N\right\|_{T,q} N^{-\frac{\beta}{d} \gamma}\right)^q  \left\| \rho^N\right \|_{T,q}^q \nonumber \\
		&+|M^N_{t}|. \nonumber
	\end{align}
	
	By Jensen's inequality with $| \cdot |^{\frac{m}{q}}$, $\frac{m}{q} > 1$  we have for all $ r \in [0,T]$ 
	\begin{align}
		\mathbb{E}\left( \sup_{t \in [0,r]}\left\|\rho_t - \rho^N_t \right\|_{q} \right)^m &\le C \, \mathbb{E}\left( \left\|\rho_0 -  \rho^N_0 \right\|_{q} \right)^m \nonumber \\ 
		&+C\int_{0}^{r}\mathbb{E}\left( \sup_{t \in [0,s]}\left\|\rho_t - \rho^N_t \right\|_{q} \right)^m ds\nonumber\\
		&+ C N^{\left(-\frac{1}{2}+\beta(1 + \frac{1 }{d}-\frac{1}{q})\right)\, m} \, \nonumber \\ 
		&+ C\left(N^{-\frac{\beta}{d}}\right)^m\mathbb{E}\| \rho^N \|_{T,q}^m \nonumber \\
		&+ C\left( N^{-\frac{\beta}{d} \gamma} \right)^m
		\mathbb{E}\left(\left \| \rho^N \right\|_{T,q}\left\|\rho^N\right\|_{T,q} \right)^m \nonumber\\
		&+C\,\mathbb{E}\left(\sup_{t\in [0,r]}|M^N_{t}| \right)^{\frac{m}{q}}.
		\label{dec 61} 	\end{align}
	Now we need a uniform in $N \in \mathbb{N}$ estimate for the regularized empirical measure, for which the proof is given in the Appendix \ref{apA} .
	\begin{lem}\label{additive estimate}
		Under the hypotheses of Theorem \ref{first main} or \ref{teo_d1}, we have the following:
		\begin{align*}
			\sup_{N \in \mathbb{N}}\Big \| \left\|\rho^N \right\|_{T,q} \Big \|_m  & < \infty.  \end{align*} 
	\end{lem}
	Finally, similar computations as those of Lemma \ref{lemma_martingal} on Appendix \ref{apA} yields
	\begin{align*}
		\mathbb{E}\left(\sup_{t\in [0,r]}\left|M^N_{t}\right| \right)^{\frac{m}{q}}
		&\lesssim    \int_{0}^{r} \mathbb{E}\left( \sup_{t \in [0,s]}\left\|\rho_t - \rho_t^N\right\|_q \right)^{m} d{s} + N^{\left(-\frac{1}{2}+\beta(1 + \frac{1 }{d}-\frac{1}{q})\right)\, m}.  \, 
	\end{align*}

	The above with  Lemma \ref{additive estimate} in (\ref{dec 61}) leads to
	\begin{align}
		\mathbb{E}\left( \sup_{t \in [0,r]}\left\|\rho_t - \rho^N_t \right\|_{q} \right)^m &\le C\,\mathbb{E}\left( \left\|\rho_0 -  \rho^N_0 \right\|_{q} \right)^m \nonumber \\
		&+C\int_{0}^{r}\mathbb{E}\left( \sup_{t \in [0,s]}\left\|\rho_t - \rho^N_t \right\|_{q} \right)^m ds\nonumber\\
		&+ C N^{\left(-\frac{1}{2}+\beta(1 + \frac{1 }{d}-\frac{1}{q})\right)\, m} \, \nonumber \\ 
		&+ C\left(N^{-\frac{\beta}{d}}\right)^m \nonumber \\
		&+ C\left(N^{-\frac{\beta}{d}\gamma}\right)^m \nonumber \\
		&+ C N^{\left(-\frac{1}{2}+\beta(1 + \frac{1 }{d}-\frac{1}{q})\right)\, m}. \, \nonumber 
	\end{align}
	Now take $r=T$. Gr\"onwall's Lemma implies
	\begin{align}
		\mathbb{E}\left( \sup_{t \in [0,T]}\left\|\rho_t - \rho^N_t \right\|_{q} \right)^m &\le C\,\mathbb{E}\left( \left\|\rho_0 -  \rho^N_0 \right\|_{q} \right)^m \nonumber \\
		&+C N^{\left(-\frac{1}{2}+\beta(1 + \frac{1 }{d}-\frac{1}{q})\right)\, m}\nonumber \\
		&+ C\left(N^{-\frac{\beta}{d}}\right)^m \nonumber \\   
		&+ C\left(N^{-\frac{\beta}{d}\gamma}\right)^m \nonumber
	\end{align}
	and since $\gamma \le 1$, we end up with
	\begin{align*}
		\Big \| \left\|\rho - \rho^N \right\|_{T,q} \Big \|_{L^m(\Omega)}  &\le C \Big \| \left\|\rho_0 - \rho^N_0 \right\|_{q} \Big \|_{L^m(\Omega)} +C N^{\left(-\frac{1}{2}+\beta(1 + \frac{1 }{d}-\frac{1}{q})\right)} + C N^{-\frac{\beta}{d} \gamma} .
		\nonumber
	\end{align*}
\end{proof}

\subsection{Proof of Theorem \ref{teo_d1}}
\begin{proof}
	Applying the Itô's formula for the $L^q$-norm of a $H_{q}^{1}$-valued process in \cite{kry} as in the proof of Theorem \ref{first main} we get
	\begin{align}
		\left\| \rho_t - \rho^N_t\right\|_{2}^{2}&=\left\| \rho_0 -\rho^N_0\right\|_{2}^{2}\nonumber+I_1 + I_2 \nonumber \\
		&- \int_{0}^{t} \int_{\mathbb{T}} \left| \nabla (\rho_s - \rho^N_s)\right|^{2} \, dx \, ds\nonumber \\
		& -\sum_{j,k=1}^d\int_0^t\sigma^{jk}_s \underbrace{\int_{\mathbb{T}} \nabla|\rho_s- \rho_s^N|^{2}\,dx }_{=0}d  B_s^k\nonumber \\
		&-M_{t}^N ,  \label{dec7}
	\end{align}
	\noindent where 
	$$I_1 \doteq 2\int_{0}^{t}\int_{\mathbb{T}}\nabla (\rho_s -\rho_s^N)\left[\rho_sF\left(x, \rho_s\right)  -\left\langle S_{s}^{N}, V^{N}(x-\cdot)F\left(\cdot,  \rho^N_s(\cdot)\right)\right\rangle \right]\,dx\,ds, $$
	$$ I_2 \doteq \frac{1}{ N^{2}} \sum_{i=1}^{N} \int_{0}^{t}\int_{\mathbb{T}}  \left| \left(\nabla V^N\right)(x-X_{s}^{i})\right|^{2} \, dx \, ds,$$
	and
	$$M^N_{t} \doteq  \frac{2}{N}\sum_{i=1}^N\int_0^t\int_{\mathbb{T}}(\rho_s - \rho_s^N)\left(\nabla V^N\right) (x - X_s^{i}) \, dx \, dW_s^{i} .$$
	
	By a change of variables  in $I_2$ and the estimate \eqref{taxa derivada V} we obtain
	\[ I_2 \le C N^{(-1+3\beta)}  \, \|\nabla V\|_{2}^2. \]

	Now, for $I_1$, we use that $F$ is Lipschitz $(\mathbf{A}^F)$ and the estimate provided by Lemma \ref{additive estimate d=2} on $\rho_s^N$ combined with the Sobolev embedding  $H_2^1(\R) \hookrightarrow C^{\frac{1}{2}}(\R)$.
	\begin{align*}
		\left| \left\langle S_s^{N}, \right. \right. & \left. \left. V^{N}(x - \cdot)\left(F\left(x, \rho_s^{N}(x)\right)-F(\cdot, \rho_s^{N}(\cdot))\right) \right \rangle \right| \\
		& \le\left\langle S_s^{N}, V^{N}(x-\cdot)\left|F\left(x, \rho_s^{N}(x)\right)-F\left(\cdot, \rho_s^{N}(\cdot)\right)\right|\right\rangle \\
		&\le L\left\langle S_s^{N}, V^{N}(x-\cdot)\left(|x - \cdot|+\left| \rho_s^{N}(x)-\rho_s^{N}(\cdot)\right|\right) \right\rangle \\
		& \le L\left\langle S_s^{N}, V^{N}(x-\cdot)\left(|x - \cdot|+\left\|\rho_s^{N}\right\|_{\frac{1}{2}}|x-\cdot|^{\frac{1}{2}}\right)\right\rangle \\
		& \le L\left(N^{-\beta}+\left\|\rho_s^{N}\right\|_{\frac{1}{2}} N^{-\frac{\beta}{2} }\right)\left|\rho_s^{N}(x)\right|
	\end{align*}
	where in the last inequality we are using that $\text{supp}\, V \subset \left\{x \in \mathbb{R}^d\mid |x|< \frac{1}{2}\right\}$. 
	
	In the same way, we add and subtract the terms $F(x, \rho_s^N)\rho_s$ and $F(x, \rho_s^N)\rho_s^N$ to obtain
	\begin{align*}
		&\left|F(x, \rho_s) \rho_s-\langle S_s^N, V^N(x-\cdot) F(\cdot, \rho_s^N)\rangle \right| \\
		&\quad \le \left|F(x, \rho_s) \rho_s- F(x, \rho_s^N) \rho_s\right| + \left|F(x, \rho_s^N) \rho_s- F(x, \rho_s^N) \rho_s^N\right| \\
		&\quad + \langle S_s^N, V^N(x-\cdot)| F(x, \rho_s^N(x)) - F(\cdot, \rho_s^N(\cdot))|\rangle \\
		&\quad\le L\left|\rho_s - \rho_s^N\right||\rho_s| + L\left|\rho_s - \rho_s^N\right| +  L\left(N^{-\beta}+\left\|\rho_s^{N}\right\|_{\frac{1}{2}} N^{-\frac{\beta}{2} }\right)\left|\rho_s^{N}\right|. 
	\end{align*}
	Now by Theorem $8.8$ in \cite{Brezis}, for all $p \geq 1$,
	\begin{align*}
		\left\| \, \cdot \,   \right\|_{\infty} \le \| \cdot \|_{p} + \|\nabla \cdot\|_{p}.
	\end{align*}
	We observe that  $ \rho \in  L^2H_2^1(\mathbb{T})\hookrightarrow L^2C^{\frac{1}{2}}(\mathbb{T})$, a.s.. Then 
	a.e. in $s \in [0,T]$, taking $p=1$, we have, a.s.
	\begin{align*}
		\|(\rho_s - \rho_s^N)^2\|_{\infty} \le \|(\rho_s - \rho_s^N)^2\|_1 + 2\|(\rho_s - \rho_s^N)\nabla(\rho_s - \rho_s^N)\|_{1}
	\end{align*}
	and thus by Holder's inequality,
	\begin{align*}
		\|\rho_s - \rho_s^N\|_{\infty}^2 \leq  \|\rho_s - \rho_s^N\|_{2}^2 + 2\|\rho_s - \rho_s^N\|_{2}\|\nabla(\rho_s - \rho_s^N)\|_{2} .
	\end{align*}
	
	By $\epsilon$-Young inequality we have 
	\begin{align*}
		& \int_{\mathbb{T}}  |\nabla (\rho_s - \rho_s^N)| \left| \rho_s - \rho_s^N \right||\rho_s| \, dx \\ &\le \, \epsilon \, \|\nabla(\rho_s - \rho_s^N)\|_{2}^2  + C_\epsilon \, \|\rho_s\|_{2}^2 \, \|\rho_s-\rho_s^N\|_{\infty}^2\\
		&\leq 2\epsilon \, \|\nabla(\rho_s - \rho_s^N)\|_{2}^2 + C_\epsilon \, \big(\|\rho_s\|_{2}^2+ 4C_\epsilon \|\rho_s\|_{2}^4 \big)\, \|\rho_s-\rho_s^N\|_{2}^2.
	\end{align*}
	
	By denoting $\hat{C}_N \doteq N^{-\beta}+\left\|\rho_s^{N}\right\|_{\frac{1}{2}} N^{-\frac{\beta}{2} }$,  the $\epsilon$-Young inequality yields
	\begin{align*}
		&|\nabla (\rho_s-\rho_s^N)||\rho_s-\rho_s^N|+\hat{C}_N|\nabla (\rho_s-\rho_s^N)||\rho_s^N| \\
		&\le 2\epsilon\, |\nabla (\rho_s-\rho_s^N)|^2 + C_\epsilon \, |\rho_s-\rho_s^N|^2 + C_\epsilon \, \hat{C}_N^2\,|\rho_s^N|^2.
	\end{align*}
	
	Therefore,  we obtain an estimate for $I_1$:
	\begin{align*}
		|I_1|&\le 2 L \int_0^t \int_{\R} |\nabla(\rho_s-\rho_s^N)| \Big(|\rho_s-\rho_s^N||\rho_s| + |\rho_s-\rho_s^N|+ \hat{C}_N|\rho_s^N|\Big) \, dx\, ds \\
		&\le \epsilon 2L\int_0^t \|\nabla (\rho_s - \rho_s^N) \|_2^2  \, ds + \, C_\epsilon \, 2L\int_0^t\big(\|\rho_s\|_{2}^2+ 4C_\epsilon \|\rho_s\|_{2}^4 \big) \|\rho_s - \rho_s^N\|_2^{2}  \, ds \\
		& +  4\epsilon L \int_0^t \|\nabla (\rho_s - \rho_s^N)\|_2^2  \,ds + C_{\epsilon} 2L \int_0^t \| \rho_s - \rho_s^N \|_2^2  \,ds   \\
		&+ C_{\epsilon}4L \, \left(N^{-2\beta}+\int_{0}^T\left \|\rho_s^{N}\right\|_{{\frac{1}{2}}}^2 N^{-2\frac{\beta}{2} }\, ds\right)  T\,\| \rho^N \|_{T,2}^2. \,
	\end{align*}
	Finally, taking  $\epsilon < \frac{1}{6L}$,  the previous estimates lead to
	\begin{align}
		\left\|\rho_t - \rho^N_t \right\|_2^{2}   &\le \left\|\rho_0 - \rho^N_0 \right\|_2^{2} + C\int_0^t (C + \|\rho\|_{T,2}^2+\|\rho\|_{T,2}^4) \|\rho_s - \rho_s^N\|_2^{2}\, ds \nonumber\\
		&+ C N^{\left(-1+3\beta\right)} \nonumber\\
		&+ C\left(N^{-2\beta}+\int_{0}^T\left \|\rho_s^{N}\right\|_{{\frac{1}{2}}}^2 N^{-\beta}\, ds\right)  \| \rho^N \|_{T,2}^2 \nonumber\\
		&+|M^N_{t}|.  \nonumber
	\end{align}

	Since $\Big\|\left\|\rho\right\|_{T,2}\Big\|_{L^{\infty}(\Omega)} < \infty$, by Jensen's inequality with $| \cdot |^{\frac{m}{2}}$, $\frac{m}{2} > 1$  we have for all $ r \in [0,T]$ 
	\begin{align}
		\mathbb{E}\left( \sup_{t \in [0,r]}\left\|\rho_t - \rho^N_t \right\|_{2} \right)^m &\le C\, \mathbb{E}\left( \left\|\rho_0 -  \rho^N_0 \right\|_{2} \right)^m \nonumber \\
		&+C\int_{0}^{r}\mathbb{E}\left( \sup_{t \in [0,s]}\left\|\rho_t - \rho^N_t \right\|_{2} \right)^m ds\nonumber\\
		&+ C N^{\left(-1+3\beta\right)\frac{m}{2}}  \,  \nonumber\\
		&+ C\left(N^{-\beta}\right)^m\mathbb{E}\| \rho^N \|_{T,2}^m \nonumber \\
		&+ C\left( N^{-\frac{\beta}{2}} \right)^m
		\mathbb{E}\left(\left \| \rho^N \right\|_{T,2}\left\|\rho^N\right\|_{
			L^2C^{\frac{1}{2}}} \right)^m  \nonumber \\&+E\left(\sup_{t\in [0,r]}|M^N_{t}| \right)^{\frac{m}{2}}. \label{dec 6}
	\end{align}
	Note that for all $ m \geq 1$,
	\begin{align}
		\sup_N\Big \| \left\|\rho^N\right\|_{L^2C^\frac{1}{2}} \,  \Big \|_m  
		\nonumber \lesssim
		\sup_N\Big \| \left\|\rho^N\right\|_{L^2H_2^1} \,  \Big \|_m  
	\end{align}
	where we use that $ L^2H_2^1(\mathbb{T})\hookrightarrow L^2C^{\frac{1}{2}}(\mathbb{T})$.  
	We now need a uniform in $N \in \mathbb{N}$ estimate for the regularized empirical measure, whose proof is given in the Appendix \ref{apA}.
	\begin{lem}\label{additive estimate d=2}
		Under the hypotheses of Theorem \ref{teo_d1}, we have the following:
		\begin{align*}
			\sup_{N \in \mathbb{N}}\Big \| \left\|\rho^N \right\|_{L^2H_2^1} \Big \|_m  & < \infty.  \end{align*} 
	\end{lem}
	Finally, similar computations as those of Lemma \ref{lemma_martingal} on Appendix \ref{apA} yield
	\begin{align*}
		\mathbb{E}\left(\sup_{t\in [0,r]}\left|M^N_{t}\right| \right)^{\frac{m}{2}}
		&\lesssim    \int_{0}^{r} \mathbb{E}\left( \sup_{t \in [0,s]}\left\|\rho_t - \rho_t^N\right\|_2 \right)^{m} d{s} + N^{\left(-1+3\beta\right)\frac{m}{2}}.  \, 
	\end{align*}
	From  (\ref{dec 6}), and  Lemmas \ref{additive estimate} and \ref{additive estimate d=2}  we conclude 
	\begin{align}
		\mathbb{E}\left( \sup_{t \in [0,r]}\left\|\rho_t - \rho^N_t \right\|_{2} \right)^m &\le C\, \mathbb{E}\left( \left\|\rho_0 -  \rho^N_0 \right\|_{2} \right)^m \nonumber \\
		&+C\int_{0}^{r}\mathbb{E}\left( \sup_{t \in [0,s]}\left\|\rho_t - \rho^N_t \right\|_{2} \right)^m ds\nonumber\\
		&+ C N^{\left(-1+3\beta\right)\frac{m}{2}} \, \nonumber \\ 
		&+ C\left(N^{-\beta}\right)^m \nonumber \\
		&+ C\left( N^{-\frac{\beta}{2}} \right)^m
		\nonumber\\
		&+ CN^{\left(-1+3\beta\right)\frac{m}{2}}. \nonumber
	\end{align}
	Then taking $r=T$, Gr\"onwall's Lemma implies
	\begin{align}
		\mathbb{E}\left( \sup_{t \in [0,T]}\left\|\rho_t - \rho^N_t \right\|_{2} \right)^m &\le C\mathbb{E}\left( \left\|\rho_0 -  \rho^N_0 \right\|_{2} \right)^m \nonumber \\
		&+ C N^{\left(-1+3\beta\right)\frac{m}{2}} + C\left( N^{-\frac{\beta}{2}} \right)^m
		\nonumber
	\end{align}
	and 
	\begin{align*}
		\Big \| \left\|\rho - \rho^N \right\|_{T,2} \Big \|_m  &\le C \Big \| \left\|\rho_0 - \rho^N_0 \right\|_{2} \Big \|_m \nonumber \\
		&+ C N^{\left(-1+3\beta\right)\frac{1}{2}} + C N^{-\frac{\beta}{2}} .
		\nonumber
	\end{align*}
	\end{proof}
	
	\appendix

	\section{Uniform in $N\in\mathbb{N}$ estimates for the regularized empirical measure} \label{apA}
	This section presents the proofs of Lemmas \ref{additive estimate} and \ref{additive estimate d=2}.
	\begin{proof}[Proof of Lemmas \ref{additive estimate} and \ref{additive estimate d=2}]
	
	Applying the It\^o's formula on $\rho_t^N$  for the function $x \mapsto |x|^q$, $q\geq2$, (essentially \eqref{dec7second} with $\rho=0$), Fubini's theorem and  its  stochastic version, we get

\begin{align}
\left\| \rho^N_t\right\|_{q}^{q}&=\left\| \rho^N_0\right\|_{q}^{q}\nonumber+I_1 + I_2\nonumber \\
&-\frac{1}{2} \int_{0}^{t} \int_{\mathbb{T}^{d}} q(q-1)\left| \rho^N_s\right|^{q-2}\left| \nabla \rho^N_s\right|^{2} \, dx \, ds\nonumber \\ &-M_{t}^N   \label{dec7prima}
\end{align}
\noindent where 
$$I_1 \doteq  q(q-1)\int_{0}^{t}\int_{\mathbb{T}^d}|\rho_s^N|^{q-2}\nabla \rho_s^N\left\langle S_{s}^{N},  V^{N}(x-\cdot) F\left(\cdot, K\ast \rho^N_s(\cdot)\right)\right\rangle dx \, ds, $$
$$ I_2 \doteq \frac{1}{2 N^{2}} \sum_{i=1}^{N} \int_{0}^{t}\int_{\mathbb{T}^{d}}  q(q-1)\left| \rho_s^N\right|^{q-2} \left| \left(\nabla V^N\right)(x-X_{s}^{i})\right|^{2} \, dx \, ds,$$
and
$$M^N_{t} \doteq \frac{1}{N} \sum_{i=1}^{N}\int_{0}^{t} \int_{\mathbb{T}^d}q\left|\rho_s^N\right|^{q-1}   \left(\nabla V^N\right)(x-X_{s}^{i}) \, dx\,dW_{s}^{i}.$$

If $q >2$ by Young inequality,  with $\frac{1}{\frac{q}{q-2}}+\frac{1}{\frac{q}{2}}=1$,  we obtain
\begin{align}
&I_2 = \frac{q(q-1)}{2 N} \sum_{i=1}^{N} \int_{0}^{t}\hspace{-3px}\int_{\mathbb{T}^{d}}  \left| \rho_s^N\right|^{q-2} \frac{1}{N}\left| \left(\nabla V^N\right)(x-X_{s}^{i})\right|^{2} \, dx \, ds \label{I2_tipo_I5} \\
&\lesssim \frac{q(q-1)}{2 N} \sum_{i=1}^{N} \int_{0}^{t}\left\| \rho_s^N\right\|_{q}^{q} ds+\frac{q(q-1)}{2 N} \sum_{i=1}^{N} \int_{0}^{t}\hspace{-3px} \int_{\mathbb{T}^{d}}\frac{1}{N^{\frac{q}{2}}}\left|\left( \nabla V^{N}\right)(x-X_{s}^{i})\right|^{q} d x \, d s.   \nonumber
\end{align}
Now for $q\geq 2$, by a change of variables we have 
\begin{align}
\nonumber \int_{0}^{t} \int_{\mathbb{T}^{d}}\left| \left(\nabla V^{N}\right)(x)\right|^{q} d x \, d s  &\le T \int_{\mathbb{R}^{d}} \left| \nabla \left[N^{\beta}V\left(N^{\beta/d}\cdot \right)\right](x)\right|^{q} d x \nonumber \\
&= T \int_{\mathbb{R}^{d}} \left| N^{\beta(1 + \frac{1 }{d})}(\nabla V)\left(N^{\beta/d}x \right)\right|^{q} d x \nonumber \\
&= T  N^{q\beta(1 + \frac{1 }{d})}\int_{\mathbb{R}^{d}} \left|(\nabla V)\left(N^{\beta/d}x \right)\right|^{q} d x \nonumber\\
&= T  N^{q\beta(1 + \frac{1 }{d})}\int_{\mathbb{R}^{d}} \left|(\nabla V)\left(N^{\beta/d}x \right)\right|^{q}\frac{N^\beta}{N^\beta} d x\nonumber \\
&= T  N^{q\beta(1 + \frac{1 }{d})}\int_{\mathbb{R}^{d}} \left|(\nabla V)\left(y \right)\right|^{q}\frac{1}{N^\beta} d y\nonumber \\
& = T N^{\left(q\beta(1 + \frac{1 }{d})-\beta\right)} \int_{\mathbb{R}^{d}}  \, \left| \nabla V(y)\right|^{q} dy \nonumber\\
& = T N^{\left(q\beta(1 + \frac{1 }{d})-\beta\right)}  \, \|\nabla V\|_{q}^q. \label{taxa derivada V}
\end{align}
Then 
\begin{align*}
I_2 &\le C_2\int_{0}^{t}\left\|  \rho^N_s\right\|_{q}^{q} d s+\frac{q(q-1)}{2}T N^{\left(-\frac{q}{2}+q\beta(1 + \frac{1 }{d})-\beta\right)}  \, \|\nabla V\|_{q}^q
\end{align*}
where $C_2 =\frac{q(q-1)}{2}$, if $q > 2$ and $C_2=0$, if $q=2$.

Now, we observe 
\begin{align*}
\left| \left\langle S_{s}^{N}, V^{N}(x-\cdot) F\left(\cdot, K\ast \rho_s^N\right)\right\rangle \right|&\le  \left\langle S_{s}^{N},\left|V^{N}(x-\cdot)F\left(\cdot, K\ast \rho_s^N\right)   \right|\right\rangle \\
&\stackrel{V \ge 0}{\le} \|F\|_{\infty} \left|\left\langle S_{s}^{N}, V^{N}(x-\cdot) \right\rangle\right| \\
&= \|F\|_{\infty} |\rho_s^N |.
\end{align*}

Then we deduce 
\begin{align*}
I_1 &= \int_{0}^{t} \int_{\mathbb{T}^d}   q(q-1) \left| \rho_s^N\right|^{q-2} \nabla \rho_s^N  \,  \left\langle S_{s}^{N}, V^{N}(x-\cdot) F\left(\cdot, K\ast \rho_s^N\right)\right\rangle dx\,  ds\\
&\le \| F\|_{\infty} \int_{0}^{t} \int_{\mathbb{T}^d}   q(q-1) \left| \rho_s^N\right|^{q-2} | \nabla \rho_s^N | \,  |\rho_s^N| \, dx \, d s.
\end{align*}

By $\epsilon$-Young inequality with $\frac{1}{2}+\frac{1}{2}=1$ we get 
\begin{align}
I_1 &\le C_{\epsilon} q(q-1)\|F\|^2_{\infty}\int_{0}^{t}\left\|  \rho^N_s\right\|_{q}^{q} ds \nonumber \\
&+\epsilon \int_{0}^{t} \int_{\mathbb{T}^d}   q(q-1) \left| \rho_s^N\right|^{q-2} | \nabla \rho_s^N |^2 \,  dx \,ds. \label{I1}
\end{align}

Thus, we join the estimates with $\epsilon < \frac{1}{2}$ and continue from \eqref{dec7prima}:
\begin{align}
\left\| \rho^N_t \right\|_q^{q}   &\le \left\| \rho^N_0 \right\|_q^{q} \nonumber \\
&+\left(C_{\epsilon}q(q-1)\|F\|_{\infty}^2 + \frac{q(q-1)}{2} 1_{[q>2]}\right)\int_{0}^{t}\left\| \rho_s^N\right\|_{q}^{q} ds\nonumber\\
&+ \frac{q(q-1)}{2}T N^{\left(-\frac{q}{2}+q\beta(1 + \frac{1 }{d})-\beta\right)} \, \|\nabla V\|_{q}^q \nonumber \\ 
&+\left|M^N_{t}\right|. \label{eps}
\end{align}

Taking the supremum, applying Jensen's inequality, and then taking the expectation, we obtain: 
\begin{align}
\mathbb{E}\left( \sup_{t \in [0,r]}\left\| \rho^N_t \right\|_q \right)^m &\le C\mathbb{E}\left( \left\| \rho^N_0 \right\|_q \right)^m \nonumber \\
&\hspace{-45px}+C\left(C_{\epsilon}q(q-1)\|F\|_{\infty}^2 + \frac{q(q-1)}{2} 1_{[q>2]}\right)^{\frac{m}{q}}\int_{0}^{r}\mathbb{E}\left( \sup_{t \in [0,s]}\left\| \rho^N_t \right\|_q \right)^m ds\nonumber\\
&\hspace{-45px}+ C\Big(\frac{q(q-1)}{2}T\Big)^{\frac{m}{q}} N^{\left(-\frac{1}{2}+\beta(1+\frac{1}{d} -\frac{1}{q})\right)m}  \, \|\nabla V\|_{q}^m \nonumber \\ 
&\hspace{-45px}+C\mathbb{E}\left(\sup_{t\in [0,r]}\left|M^N_{t}\right| \right)^{\frac{m}{q}}. \label{dec 6prima}
\end{align}
Finally,  from (\ref{dec 6prima}) and  lemma \ref{lemma_martingal}  we deduce 
\begin{align}
\mathbb{E}\left( \sup_{t \in [0,r]}\left\| \rho^N_t \right\|_q \right)^m &\le  C\mathbb{E}\left( \left\| \rho^N_0 \right\|_q \right)^m +C\int_{0}^{r}\mathbb{E}\left( \sup_{t \in [0,s]}\left\| \rho^N_t \right\|_q \right)^m ds\nonumber\\
&+ C\left(N^{\left(-\frac{1}{2}+\beta(1+\frac{1}{d} -\frac{1}{q})\right)m} +N^{\left(-\frac{1}{2}+\beta(1 + \frac{1 }{d}-\frac{1}{q})\right)m}  \right). \nonumber
\end{align}
Taking  $r=T$ and  Gr\"onwall's Lemma, we conclude that 
\begin{align}
\mathbb{E}\left( \sup_{t \in [0,T]}\left\| \rho^N_t \right\|_q \right)^m &\le  C\mathbb{E}\left( \left\| \rho^N_0 \right\|_q \right)^m \nonumber + C N^{\left(-\frac{1}{2}+\beta(1+\frac{1}{d} -\frac{1}{q})\right)m}  \nonumber
\end{align}
and 
\begin{align*}
\Big \| \left\|\rho^N \right\|_{T,q} \Big \|_m  &\le C \Big \| \left\|\rho^N_0 \right\|_{q} \Big \|_m \nonumber + CN^{-\frac{1}{2}+\beta(1+\frac{1}{d} -\frac{1}{q})}.  \nonumber 
\end{align*}

\noindent This proves Lemma \ref{additive estimate}.
In order to get a proof of  Lemma \ref{additive estimate d=2}, considering the above estimates, and continuing from \eqref{dec7prima} with $\epsilon=\frac{1}{4}$, instead of \eqref{eps}, we obtain

\begin{align*}
\left\| \rho^N_t \right\|_2^{2}   +\frac{1}{2} \int_{0}^{t}  \left\| \nabla \rho^N_s\right\|_2^{2} \, ds &\le \left\| \rho^N_0 \right\|_2^{2} +C\int_{0}^{t}\left\| \rho_s^N\right\|_{2}^{2} ds\\
&+ T N^{\left(-1+2\beta(1 + \frac{1 }{d})-\beta\right)} \, \|\nabla V\|_{2}^2.  \\ 
&+\left|M^N_{t}\right|.
\end{align*}

\noindent  From  Lemma \ref{additive estimate} we have  $\sup_{N \in \mathbb{N}}\mathbb{E}\|\rho^N\|_{L^2L^2}   <\infty$.  Now we estimate  $\|\nabla \rho^N\|_2$.
By Jensen's inequality, with $| \cdot |^{\frac{m}{2}}$ and  $m > 2$,   we have 
\begin{align}
\mathbb{E}\left(\int_{0}^{r}  \left\| \nabla \rho^N_s\right\|_2^{2} \, ds\right)^{\frac{m}{2}} \nonumber &\le  C\Big(1+E \left( \left\| \rho^N_0 \right\|_2 \right)^m \Big)\nonumber \\
&+C \int_{0}^{r}\mathbb{E}\left( \sup_{t \in [0,s]}\left\| \rho^N_t \right\|_2 \right)^m ds\nonumber\\
&+C\mathbb{E}\left(\sup_{t\in [0,r]}\left|M^N_{t}\right| \right)^{\frac{m}{2}}. \nonumber
\end{align}

From  Lemma \ref{lemma_martingal}  we conclude

\begin{align}
\sup_N \left(\mathbb{E}\left(\int_{0}^{T}  \left\| \nabla \rho^N_s\right\|_2^{2} \, ds\right)^{\frac{m}{2}} \right)^{\frac{1}{m}}\nonumber &\lesssim  1 + \sup_N \Big \| \left\| \rho^N_0 \right\|_2\Big \|_m \nonumber \\
&\quad + T \sup_N \Big \| \left\| \rho^N \right\|_{T,2} \Big \|_m < \infty. \nonumber
\end{align}

\end{proof}

\begin{lem}\label{lemma_martingal} Under the assumptions of Theorem \ref{first main} or \ref{teo_d1}, it holds that
\begin{align*}
\mathbb{E}\left(\sup_{t\in [0,r]}\left|M^N_{t}\right| \right)^{\frac{m}{q}}
&\lesssim    \int_{0}^{r} \mathbb{E}\left( \sup_{t \in [0,s]}\| \rho_t^N\|_q \right)^{m} d{s} + N^{\left(-\frac{1}{2}+\beta(1 + \frac{1 }{d}-\frac{1}{q})\right)m}.
\end{align*}
\end{lem}

\begin{proof}
First, notice that
\begin{align*}
M^N_{t} \doteq \int_{0}^{t} \int_{\mathbb{T}^d}q\left|\rho_s^N\right|^{q-1}   \frac{1}{N} \sum_{i=1}^{N} ( \nabla V^{N})(x-X_{s}^{i}) \, dx\, dW_{s}^{i}
\end{align*}
is a martingale, and $\left|M^N_{t}\right|$ is a positive sub-martingale. So, by Doob's maximal inequality
\begin{align*}
\mathbb{E}\left(\sup_{t\in [0,r]}\left|M^N_{t}\right| \right)^{\frac{m}{q}} \leq \left( \frac{\frac{m}{q}}{\frac{m}{q} - 1} \right)^{\frac{m}{q}}\mathbb{E}\left|M^N_{r}\right|^{\frac{m}{q}}.
\end{align*}
Now, by Burkholder-Davis-Gundy inequality
\begin{align*}\mathbb{E}\left|M^N_{r}\right|^{\frac{m}{q}} &\doteq\mathbb{E}\left|\frac{1}{N} \sum_{i=1}^{N}\int_{0}^{r} \int_{\mathbb{T}^d}q\left|\rho_s^N\right|^{q-1}   ( \nabla V^{N})(x-X_{s}^{i}) \, dx \, dW_{s}^{i} \right|^{\frac{m}{q}}\\
&\lesssim\mathbb{E}\left|\frac{1}{N^{2}} \sum_{i=1}^{N}\int_{0}^{r} \left(\int_{\mathbb{T}^d}q\left|\rho_s^N\right|^{q-1}  ( \nabla V^{N})(x-X_{s}^{i})\, dx\right)^2 ds \right|^{\frac{m}{2q}}\\
&=\mathbb{E}\left|\frac{1}{N} \sum_{i=1}^{N}\int_{0}^{r} \left(\int_{\mathbb{T}^d}q\left|\rho_s^N\right|^{q-1}  \frac{1}{N^{\frac{1}{2}}}( \nabla V^{N})(x-X_{s}^{i})\, dx\right)^2 ds \right|^{\frac{m}{2q}}
\end{align*}
where in the last equality we moved a power of $N$ inside the integral. Now by Young's inequality with $\frac{1}{\frac{q}{q-1}} + \frac{1}{q}$ and Jensen's inequality,
\begin{align*}
E|M^N_{r}|^{\frac{m}{q}}&\lesssim\mathbb{E}\left|\frac{1}{N} \sum_{i=1}^{N}\int_{0}^{r} \left(\int_{\mathbb{T}^d}\left|\rho_s^N\right|^{q}  dx\right)^2 ds \right|^{\frac{m}{2q}}\\
&+\mathbb{E}\left|\frac{1}{N} \sum_{i=1}^{N}\int_{0}^{r} \left(\int_{\mathbb{T}^d}\frac{1}{N^{q/2}}|( \nabla V^{N})(x-X_{s}^{i})|^q  \, dx\right)^2 ds \right|^{\frac{m}{2q}}\\
&\lesssim\mathbb{E}\int_{0}^{r} \|\rho_s^N\|_q^m \, ds \\
&+ \frac{1}{N^{\frac{m}{2}}}\left| \int_{0}^{r} \left(\int_{\mathbb{T}^d}|( \nabla V^{N})(x)|^q \, dx\right)^2 ds \right|^{\frac{m}{2q}}
\end{align*}
where in the last inequality, we have applied a change of variables. Finally, by \eqref{taxa derivada V}
\begin{align*}
E\left|M^N_{r}\right|^{\frac{m}{q}}&\lesssim\mathbb{E}\int_{0}^{r} \left\|\rho_s^N\right\|_q^m ds \\
&+ N^{-\frac{m}{2}} \Big(T N^{\left(2q\beta(1 + \frac{1 }{d})-2\beta\right)}  \, \|\nabla V\|_{q}^{2q}\Big)^{\frac{m}{2q}}\\
&\lesssim\mathbb{E}\int_{0}^{r} \left\|\rho_s^N\right\|_q^m ds \\
&+ N^{\left(-q+2q\beta(1 + \frac{1 }{d})-2\beta\right)\frac{m}{2q}}  \, \|\nabla V\|_{q}^{m}\\
&\lesssim    \int_{0}^{r} \mathbb{E}\left( \sup_{t \in [0,s]}\left\| \rho_t^N\right\|_q \right)^{m} d{s}\nonumber\\
&+ N^{\left(-\frac{1}{2}+\beta(1 + \frac{1 }{d}-\frac{1}{q})\right)\,m}  \, \|\nabla V\|_{q}^{m}.
\end{align*}
\end{proof}

\section{Well-posedness of limiting equation (\ref{SPDE_Ito})} \label{apC}

\medskip

Here, we will prove the well posedness of the limiting equation for the system of particles in a suitable space. We will adapt the strategy of \cite{Hi}.

\begin{proof} [Proof of Theorem \ref{Teo_Krylov}]
The proof is divided into two parts. First, we solve the linearized problem using the $L^q$-Theory of SPDEs. Next, we make a contraction argument to yield the solution of the original equation as a fixed point of the solution map.

Let \[\mathbb{B}\doteq \left\{\rho \in S_{\mathcal{F}^B}^{\infty}\left([0, T]; L^q \left( \mathbb{T}^d\right)\right) \mid \Big\|\left\|\rho\right\|_{T,q}\Big\|_{L^{\infty}(\Omega)} \le \kappa \ell \right\} \]
with metric $d(\rho, \rho^{\prime})\doteq\Big\|\|\rho-\rho^{\prime}\|_{T,q}\Big\|_{L^{\infty}(\Omega)}$,  $\kappa=\|\rho_0\|_q$ and $\ell$ is a positive constant to be determined.

We define the operator $\mathcal{T}: \mathbb{B} \rightarrow S_{\mathcal{F}^B}^{\infty}\left([0, T] ; L^q\left(\mathbb{T}^d\right)\right)$ as follows: for each $\xi \in \mathbb{B}$, let $\mathcal{T}(\xi):=\rho^{\xi}$ be the solution to the following linear SPDE:
\begin{numcases}
\mathrm{d}\rho_t = \frac{1}{2} \sum_{i, j=1}^d \partial_{i j} \left(\rho_t\sum_{k=1}^{d}\left(\nu_t^{ik}\nu_t^{jk} + \sigma_t^{i k} \sigma_t^{j k}\right)\right)dt - \nabla \cdot \big(\rho_t F(\cdot\,, K\ast \xi_t)\big) \, dt \nonumber\\ \hspace{28px} - \nabla \rho_t \cdot \sigma_t\, d  B_t \label{lin_SPDE} \\
\rho_0 \text { is given. }  \nonumber
\end{numcases}

Now by Assumption $(\boldsymbol{A}^c)$, we rewrite $(\ref{lin_SPDE})$ as follows:
\begin{numcases}
\mathrm{d}\rho_t = \frac{1}{2} \sum_{i, j=1}^d \partial_{i j} \rho_t\sum_{k=1}^{d}\left(\nu_t^{ik}\nu_t^{jk} + \sigma_t^{i k} \sigma_t^{j k}\right)dt + f_t(\rho) \, dt \nonumber\\ \hspace{28px} - \nabla \rho_t \cdot \sigma_t\, d  B_t \label{lin_SPDE 1} \\
\rho_0 \text { is given. }  \nonumber
\end{numcases}
where
\begin{align*} f_t(\rho) &\doteq -\nabla \cdot (\rho_tF(\cdot,K\ast \xi_t)) +\sum_{i,j=1}^d \, \partial_j \rho_t \, \partial_i \Big((\nu\nu^{\top})^{ij}_t + (\sigma\sigma^{\top})^{ij}_t\Big) \\
&+ \frac{1}{2} \, \rho_t \sum_{i,j=1}^d \partial_{ij} \Big((\nu\nu^{\top})^{ij}_t + (\sigma\sigma^{\top})^{ij}_t\Big) \\
&= -\nabla \cdot (\rho_tF(\cdot,K\ast \xi_t)) + 2 \sum_{j=1}^d \partial_j \rho_t \sum_{i=1}^d \partial_i a_t^{ij} + \rho_t \sum_{i,j=1}^d \partial_{ij} a_t^{ij}
\end{align*}	
and 
\begin{align*} a^{ij}_t \doteq \frac{1}{2}\sum_{k=1}^d \Big(\nu^{ik}_t \nu^{jk}_t + \sigma^{ik}_t\sigma^{jk}_t\Big) \label{aij}.
\end{align*}

We now proceed to verify the conditions 5.1–6 of Theorem 5.1 in \cite{kry}. For the reader's convenience, we restate these conditions below using our notations:

\begin{itemize}
\item[5.1] There exist $\Lambda, \lambda>0$ such that
\[  \Lambda |\xi|^2 \geq \sum_{i,j,k=1}^d \nu_t^{ik}(x) \nu_t^{jk}(x)\, \xi^i\xi^j \ge \lambda \, |\xi|^2, \]
for any $t\geq0$, $x \in \mathbb{T}^d$, $\xi \in \mathbb{R}^d$.

\item[5.2] For any $i,j\in \{1,...,d\}$ and $\epsilon >0$, there exists $\delta > 0$, such that
\[|a^{ij}_t(x) - a^{ij}_t(y)| + |\sigma^{i}_t(x) - \sigma^{i}_t(y)| < \epsilon \]
whenever $|x-y| < \delta$, $t\geq0$.

\item[5.3] For any  $i,j \in \{1,...,d\}$ and $t\geq0$, the functions $a^{ij}_t$ and $\sigma^i_t$ are continuously differentiable.

\item[5.4] For any $u \in H^{1}_q$, $f_t(u)$ is predictable as a function taking values in $H^{-1}_q$.

\item[5.5] For any $i,j \in \{1,...,d\}$ and $t\geq0$,
\[  \|a^{ij}_t\|_{C^1} +   \|\sigma^{i}_t\|_{\infty} \leq M  . \]

\item[5.6] For any $\epsilon > 0$, there exists $C_{\epsilon} > 0$, such that, for any $t\geq0$ and $u_t, v_t \in H^1_q$, we have

\[\left\|f_t(u) - f_t(v)\right\|_{-1,q} \leq \epsilon \left\|u_t - v_t\right\|_{1,q} + C_{\epsilon}\left\|u_t - v_t\right\|_{-1,q}.\] 
\end{itemize}

Now note that the condition 5.1 is exactly $(\mathbf{A}^c_{iii})$.
The conditions 5.2, 5.3 and 5.5 are verified by assumption $(\mathbf{A}^c_i)$.  
For condition 5.4, 
note that $f_t(u)$ is a composition of predictable functions. 

For the remaining one, 5.6, we notice that if $F$ is given by $(\mathbf{A}^F$) (thus bounded) and $K$ by $(\mathbf{A}^K)$, we have for all $u_t, v_t \in {H}^1_q $,
\begin{align*}
\left\|\nabla \cdot( F(\cdot,  K\ast \xi_t)(u_t - v_t) \right\|_{-1,q} &=
\left\|(I-\Delta)^{-\frac{1}{2}} \nabla \cdot( F(\cdot,  K\ast \xi_t)(u_t - v_t) )\right\|_{q} \\ & \le C\left\|F(\cdot,K\ast \xi_t) (u_t - v_t)\right\|_{q} \\ 
& \le C\left\|u_t - v_t\right\|_{q}\\ 
& \le C\left\|u_t - v_t\right\|_{1,q}^{\frac{1}{2}}\left\|u_t - v_t\right\|_{-1,q}^{\frac{1}{2}}\\ 
& \le \epsilon\left\|u_t - v_t\right\|_{1,q} + \epsilon^{-1} C^2 \left\|u_t - v_t\right\|_{-1,q},
\end{align*}
by the interpolation inequality. On the other hand, if we assume $(\mathbf{A}^I)$ and $(\mathbf{A}^K)$ we have for all $u_t, v_t \in {H}^1_q $,
\begin{align*}
\left\|\nabla \cdot\big ( F(\cdot,  K\ast \xi_t)(u_t - v_t) \big)\right\|_{-1,q} &=
\left\|(I-\Delta)^{-\frac{1}{2}} \nabla \cdot\big(K\ast \xi_t)(u_t - v_t)\big)\right\|_{q} \\ & \le C\left\|(K\ast \xi_t) (u_t - v_t)\right\|_{q} \\ 
& \le C C_K \left\|\xi_t\right\|_{q}\left\|u_t - v_t\right\|_{q}\\ 
& \le C C_K \kappa l\left\|u_t - v_t\right\|_{0,q}\\ 
& \le c \left\|u_t - v_t\right\|_{1,q}^{\frac{1}{2}}\left\|u_t - v_t\right\|_{-1,q}^{\frac{1}{2}}\\ 
& \le \epsilon\left\|u_t - v_t\right\|_{1,q} +  \epsilon^{-1}c^2\left\|u_t - v_t\right\|_{-1,q},
\end{align*}
by the interpolation inequality and using $ \xi \in \mathbb{B}$. Recall that $\|K\ast f\|_{\infty} \le C_K \|f \|_{q}$, for $q$ given by $(\textbf{A}^K)$.

For the remaining terms in $f_t(\rho)$, we have 
\begin{align*}
\| (u_t-v_t) \, \partial_{ij} a_t^{ij} \|_{-1,q} \le \| (u_t-v_t) \, \partial_{ij} a_t^{ij} \|_{q} \le \|a_t^{ij}\|_{C^2} \| u_t-v_t\|_{q}
\end{align*}
and then it follows by interpolation inequality as before, while for the remaining one, we use the dual space $(H^{-1}_q)^*=H^1_{q^\prime}$ with the duality pair given by $(\cdot, \cdot)_{-1,1}$. Indeed, by applying the Leibniz's rule, we have	
\begin{align*}
&\big \| \, \partial_j (u_t-v_t) \, \partial_i a_t^{ij} \big \|_{-1,q} \\ &= \sup_{\|\phi\|_{1,q^\prime}=1} \big( \partial_j (u_t-v_t)\,  \partial_i a_t^{ij} ,\phi\big)_{-1,1} \\
&=  \sup_{\|\phi\|_{1,q^\prime}=1} \big(   (u_t-v_t)\,  \partial_i a_t^{ij}, - \partial_j \phi \big)_{-1,1}  + \sup_{\|\phi\|_{1,q^\prime}=1} \big(   (u_t-v_t)\,  \partial_{ji} a_t^{ij} ,- \phi\big)_{-1,1} \\
&\le  \sup_{\|\phi\|_{1,q^\prime}=1} \left( \|\partial_j \phi\|_{q^\prime} \, \|(u_t-v_t)\,  \partial_i a_t^{ij}\|_q  + \|\phi\|_{1,q'}\| (u_t-v_t) \, \partial_{ij} a_t^{ij} \|_{q} \right) \\
&\le 2 \, \|a_t^{ij}\|_{C^2} \| u_t-v_t\|_{q}
\end{align*}
\noindent where the H\"older's inequality was applied in the second-to-last inequality.  Then again, interpolation inequality yields the estimate that verifies assumption 5.6.

Notice that the initial condition $\rho_0 \in H^{1-\frac{2}{q}}_q\left(\mathbb{T}^d\right)$ is deterministic, therefore, we have by Theorem 5.1 in \cite{kry} that the linear SPDE (\ref{lin_SPDE}) admits a unique solution $\rho^{\xi} \in L_{\mathcal{F}^B}^q\left([0,T];H^{1}_q\left(\mathbb{T}^d\right)\right)$ (see  Definition 3.1 in \cite{kry}). In addition, since Theorem 7.1 (iii) in \cite{kry} holds for $q\geq 2$ (by Theorem 4.2 in \cite{kry}) and $\|\rho_t\|_{1}=\|\rho_0\|_{1}=1$, a.s., we have  $ \rho^{\xi} \in S_{\mathcal{F}^B}^{q}\left([0, T] ; L^1 \cap L^q\left(\mathbb{T}^d\right)\right)$. Moreover, by the maximum principle (Theorem 5.12 in \cite{kry}) $\rho^{\xi}(t,\cdot) \geq 0$, a.s..
\smallskip

We now check that $\rho^{\xi}$ is also in $S_{\mathcal{F}^B}^{\infty}\left([0, T] ; L^1 \cap L^q\left(\mathbb{T}^d\right)\right)$, and we drop the superscript $\xi$ to ease the computations. For the following, if not mentioned, we are assuming $(\mathbf{A}^F)$ and $(\mathbf{A}^K)$, and we make observations about the other sets of assumptions. Applying the It{\^o}'s formula for the $L^q$-norm of a $H_q^{1}$-valued process in \cite{kry},
\begin{align}
&\| \rho_t \|_{q}^q = \|\rho_0\|_{q}^q \nonumber\\&\quad+ \frac{1}{2}\int_{\mathbb{T}^d} \int_0^t q | \rho_s (x) |^{q - 2} \rho_s (x)
\big(\nu_s(x)\nu_s^{\top}(x)+\sigma_s(x)\sigma_s^{\top}(x) \big)\boldsymbol{\cdot} D^2\rho_s (x) \, ds \, dx\nonumber\\
&\quad - \int_{\mathbb{T}^d} \int_0^t q | \rho_s (x) |^{q - 2} \rho_s (x)\, f_s(\rho)(x) \, ds \, dx\nonumber\\
&\quad - \int_{\mathbb{T}^d} \int_0^t q | \rho_s (x) |^{q - 2} \rho_s (x) \nabla
\rho_s (x) \cdot \sigma_s(x) \, d  B_s \, d  x\nonumber\\
&\quad + \frac{1}{2}  \int_{\mathbb{T}^d} \int_0^t q (q - 1) | \rho_s
(x) |^{q - 2}  | \sigma^{\top}_s(x) \, \nabla \rho_s (x) |^2 \, ds \, dx \nonumber\\
&= \text{(I) + (II) + (III) + (IV)}. \label{ito rho}
\end{align}

Above, $A\boldsymbol{\cdot}B \doteq \sum_{i,j} A_{ij}B_{ij}$. Using stochastic Fubini and integration by parts, we handle each term separately. Notice that under assumption $(\mathbf{A}^c_{ii})$, i.e., $\sum_i \partial_i \sigma_t^{ik} =0$ we have
\begin{align*}
\text{(III)} &= - \int_0^t \int_{\mathbb{T}^d}  \sum_{ik} q | \rho_s (x) |^{q - 2} \rho_s (x) \, \partial_i
\rho_s (x)\sigma^{ik}_s(x) \,d  x\,  d  B^k_s \,  \\
&= - \int_0^t \int_{\mathbb{T}^d}  \sum_{ik} \partial_i (| \rho_s (x) |^{q}) \,   \sigma^{ik}_s(x) \,dx \, d  B^k_s \\
&= \int_0^t \int_{\mathbb{T}^d}   | \rho_s (x) |^{q} \,  \sum_{ik} \partial_i\sigma^{ik}_s(x) \, dx \, d  B^k_s  = 0.\\
\end{align*}
For (II), we have two parts
\begin{align*}
\text{(II)} &= - \int_0^t \int_{\mathbb{T}^d}   q | \rho_s (x) |^{q - 2} \rho_s (x) \nabla
\cdot [\rho_s (x) F(x,K\ast \xi_s)] \, dx \, ds \\
& - \int_0^t \int_{\mathbb{T}^d}   q | \rho_s (x) |^{q - 2} \rho_s (x) \left(2 \sum_{ij} \partial_{j} \rho_s \, \partial_i a^{ij}_s + \rho_s \sum_{ij} \partial_{ij} a^{ij}_s  \right)  \, dx \, ds
\end{align*}
For the first one, after integration by parts, we use $\epsilon$-Young inequality with $\frac{1}{2}+\frac{1}{2}=1$
\begin{align}
&= q (q - 1)   \int_0^t \int_{\mathbb{T}^d} |\rho_s (x) |^{q - 2}  \rho_s(x) \nabla \rho_s(x) \cdot F(x,K\ast \xi_s )\, dx \, ds \nonumber \\
&\le q(q-1)L   \int_0^t \int_{\mathbb{T}^d} |\rho_s (x) |^{q - 2} \rho_s(x) |\nabla \rho_s(x)| \, dx \, ds \label{pgradp}\\
&\le q(q-1)L \, \int_0^t \int_{\Rd}  |\rho_s(x)|^{q-2} \big(C_\epsilon \, \rho_s(x)^2 + \epsilon\, |\nabla \rho_s(x)|^2\big) \, dx \, ds \nonumber \\
&= q(q-1)LC_\epsilon  \int_0^t \|\rho_s\|_{q}^{q} \, ds + \epsilon \, q(q-1)L \int_0^t \int_{\Rd} |\rho_s|^{q-2}|\nabla \rho_s|^2 \, dx \, ds \nonumber
\end{align}

\noindent Above, we used $(\mathbf{A}^F)$ so that $F$ is bounded by $L$. In fact, if one assumes $(\mathbf{A}^I)$ and $(\mathbf{A}^K)$, $(\ref{pgradp})$ implies that the first term in $\text{(II)}$ reads
\begin{align*}
\leq C_\epsilon q(q-1) C_K^2\kappa^2 \ell^2 \int_0^t \|\rho_s\|_{q}^{q} \, ds + \epsilon \, q(q-1) \int_0^t \int_{\Rd} |\rho_s|^{q-2}|\nabla \rho_s|^2 \, dx \, ds,
\end{align*}
since $\|K \ast \xi_s\|_{L^\infty} \le C_K \|\xi_s\|_{q} \le C_K\kappa\ell$. For the second $\text{(II)}$,
\begin{align*}
& - \int_0^t \int_{\mathbb{T}^d}   q | \rho_s (x) |^{q - 2} \rho_s (x) \Big(2 \sum_{ij} \partial_{j} \rho_s \, \partial_i a^{ij}_s + \rho_s \sum_{ij} \partial_{ij} a^{ij}_s  \Big)  \, dx \, ds \\
& \le \int_0^t \int_{\mathbb{T}^d}  2q | \rho_s|^{q-2} \rho_s \Big| \sum_{ij}( \partial_{j} \rho_s \, \partial_i a^{ij}_s) \Big| \, dx \, ds + q  \int_0^t  \sum_{ij} \|a_s^{ij}\|_{C^2}\, \| \rho_s\|_q^{q} \, ds \\
& \le \int_0^t \int_{\mathbb{T}^d}  2q | \rho_s |^{q-2}\rho_s \, |\nabla \rho_s| \,\Big(\sum_j \big(\sum_i \partial_i a^{ij}_s \big)^2\Big)^{\frac{1}{2}}  \, dx \, ds + qM \int_0^t  \| \rho_s \|_q^{q} \, ds
\end{align*}
\noindent where we used Cauchy-Schwarz inequality. In order to apply the Gr{\"o}nwall's lemma, the second term above is already suitable, while the first one, after bounding the $a^{ij}$ term with its $L^\infty$ norm provided in $(\textbf{A}^c)$, is a multiple of \eqref{pgradp}, from which we get an analogous estimate. In fact, due to assumption $(\mathbf{A}^c_i)$, we have
\begin{align*}
	&\Big|\sum_{ik} \partial_i\nu^{ik}_s \nu^{jk}_s + \nu^{ik}_s \partial_i\nu^{jk}_s\Big|
	\\ &\le \sum_{ik} \Big( \|\partial_i\nu^{ik}_s\|_{\infty}  + \|\nu^{ik}_s\|_{\infty} \Big) \Big( \|\nu^{jk}_s\|_{\infty} + \| \nabla\nu^{jk}_s\|_{\infty} \Big) \le 2M^2.\end{align*}
\noindent and analogously for $\sigma$. Thus,
\[  \text{(II)} \le C \int_0^t \|\rho_s\|_q^q \, ds + \epsilon \, (q(q-1)L+q\sqrt{d}\, 4M^2) \int_0^t \int_{\Rd} |\rho_s|^{q-2} |\nabla \rho_s|^2 \, dx\, ds. \]

Also, if $(\mathbf{A}^I)$ and $(\mathbf{A}^K)$ are in place
\begin{align*}  \text{(II)} &\le C_\epsilon\left(  q(q-1) C_K^2\kappa^2 \ell^2 + q\sqrt{d}\, 4M^2 + qM \right)\int_0^t \|\rho_s\|_q^q \, ds\\
&+ \epsilon \, (q(q-1)+q\sqrt{d}\, 4M^2) \int_0^t \int_{\Rd} |\rho_s|^{q-2} |\nabla \rho_s|^2 \, dx\, ds. \end{align*}

Now for (I), again by integration by parts,
\begin{align*}
\text{(I)} &\doteq\frac{1}{2}  \int_0^t \int_{\mathbb{T}^d} \sum_{ijk} q | \rho_s|^{q - 2} \rho_s \,(  \nu^{ik}_s\nu^{jk}_s+ \sigma^{ik}_s\sigma^{jk}_s)\,   \partial_{ij} \rho_s   \, dx \, ds\\
&=-\frac{1}{2} \int_0^t \int_{\mathbb{T}^d} \sum_{ijk} q(q-1) | \rho_s   |^{q - 2} \partial_i\rho_s   \, (  \nu^{ik}_s\nu^{jk}_s+ \sigma^{ik}_s\sigma^{jk}_s) \, \partial_{j} \rho_s   \, dx \, ds\\
&\quad -\frac{1}{2} \int_0^t \int_{\mathbb{T}^d} \sum_{ijk} q | \rho_s   |^{q - 2} \rho_s\,    \partial_i\Big(  \nu^{ik}_s\nu^{jk}_s+ \sigma^{ik}_s\sigma^{jk}_s\Big) \partial_{j} \rho_s   \, dx \, ds.
\end{align*}
Notice that $\sum_{ijk} \nu^{ik}_s \nu^{jk}_s \, \partial_i \rho_s \, \partial_j \rho_s \ge \lambda \, |\nabla \rho_s|^2$, by assumption $(\textbf{A}^c_{iii})$. So by multiplying this inequality by $-1$ and taking into account that
\begin{align*}
\sum_{k}    \Big|\sum_{i} \partial_i\rho_s\sigma^{ik}_s\Big|^2 = \sum_{ijk}   \partial_i\rho_s   \, ( \sigma^{ik}_s\sigma^{jk}_s) \, \partial_{j} \rho_s,
\end{align*}
we get 
\begin{align*}
\text{(I)}&\le -\frac{1}{2} \int_0^t \int_{\mathbb{T}^d}  q(q-1) | \rho_s   |^{q - 2} \, \lambda \, |\nabla \rho_s|^2 \, dx \, ds \\
&\quad  -\frac{1}{2} \int_0^t \int_{\mathbb{T}^d}  \sum_{k} q(q-1) | \rho_s   |^{q - 2}  \Big|\sum_{i} \partial_i\rho_s\sigma^{ik}_s\Big|^2 \, dx \, ds\\
&\quad -\frac{1}{2} \int_0^t \int_{\mathbb{T}^d} q | \rho_s   |^{q - 2} \rho_s\, \sum_j \partial_{j} \rho_s \,  \partial_i\Big(\sum_{ik} \nu^{ik}_s \nu^{jk}_s+ \sigma^{ik}_s \sigma^{jk}_s \Big)    \, dx \, ds\\
&\le -\frac{1}{2} \int_0^t \int_{\mathbb{T}^d} q(q-1) | \rho_s   |^{q - 2} \Big(\lambda \, | \nabla \rho_s|^2+|\sigma^T_s \nabla \rho_s|^2\Big)\, dx \, ds\\
&\quad +\frac{1}{2} \int_0^t \int_{\mathbb{T}^d} q | \rho_s   |^{q - 1}\, |\nabla \rho_s| \Big(\sum_j \Big(\partial_i\sum_{ik} (\nu^{ik}_s \nu^{jk}_s+ \sigma^{ik}_s \sigma^{jk}_s )\Big)^2\Big)^{\frac{1}{2}} dx \, ds
\end{align*}
\noindent where the last inequality follows by Cauchy-Schwarz inequality. Notice that the second-to-last line contains -(IV). Recall 
\begin{align*}
&\Big(\sum_j \Big(\partial_i\sum_{ik} (\nu^{ik}_s \nu^{jk}_s+ \sigma^{ik}_s \sigma^{jk}_s )\Big)^2\Big)^{\frac{1}{2}} \le \Big( \sum_j (4M^2)^2 \Big)^{\frac{1}{2}} = \sqrt{d} \, 4M^2.\end{align*}
Thus, we have
\begin{align*}
\text{(I)} &\le -\text{(IV)} -\frac{\lambda}{2} \int_0^t \int_{\mathbb{T}^d} q(q-1) | \rho_s   |^{q - 2}\, |\nabla \rho_s|^2 \,  dx \, ds    \\
&\quad +\frac{1}{2} \int_0^t \int_{\mathbb{T}^d} q | \rho_s   |^{q - 1}\, |\nabla \rho_s| \, 4M^2 \sqrt{d}\,  dx \, ds   \\
&\le -\text{(IV)} -\frac{\lambda}{2} \int_0^t \int_{\mathbb{T}^d} q(q-1) | \rho_s   |^{q - 2}\, |\nabla \rho_s|^2 \,  dx \, ds    \\
&\quad + C_\epsilon \int_0^t \|\rho_s\|_{q}^{q} \, ds + \epsilon  \, q 4 M^4 d  \int_0^t \int_{\Rd}|\rho_s|^{q-2}|\nabla \rho_s|^2 \, dx \, ds
\end{align*}
\noindent by $\epsilon$-Young inequality. Finally,
\begin{align*}
\| \rho_t \|_{q}^q &=  \|\rho_0\|_{q}^q + \text{(I) + (II) + (III) + (IV)} \\
&\le \|\rho_0\|^q_{q} + C \int_0^t \|\rho_s\|_{q}^{q} \, ds \\
&\quad + \Big(\epsilon L\, q(q-1) + \epsilon q4M^4d  -\frac{\lambda}{2} q(q-1)\Big) \int_0^t \int_{\Rd} |\rho_s|^{q-2}|\nabla \rho_s|^2 \, dx \, ds.
\end{align*}
Again, if $(\mathbf{A}^I)$ and $(\mathbf{A}^K)$ are in place
\begin{align*}
	\| \rho_t \|_{q}^q &=  \|\rho_0\|_{q}^q + \text{(I) + (II) + (III) + (IV)} \\
	&\le \|\rho_0\|^q_{q} + C_\epsilon\left(  q(q-1) C_K^2\kappa^2 \ell^2 + q\sqrt{d}4M^2 + qM + 1 \right) \int_0^t \|\rho_s\|_{q}^{q} \, ds \\
	&\quad + \Big(\epsilon q(q-1) + \epsilon q4M^4d  -\frac{\lambda}{2} q(q-1)\Big) \int_0^t \int_{\Rd} |\rho_s|^{q-2}|\nabla \rho_s|^2 \, dx \, ds.
\end{align*}

\noindent Finally, under the assumption \((\mathbf{A}^F)\), by taking \(\epsilon\) sufficiently small and \(\ell\doteq e^{\frac{CT}{q}}\), an application of Grönwall's Lemma yields
\[ \sup_{t \in [0,T]} \norm{ \rho_t}_{q} \le \norm{ \rho_0 }_{q}  \big(e^{C T}\big)^{\frac{1}{q}} \le \kappa  \ell  \,\, \text{ a.s.}\]
which means $\rho^{\xi} \in \mathbb{B}$.

Now, under the assumptions \((\mathbf{A}^I)\) and \((\mathbf{A}^K)\), and by taking $T$ and \(\epsilon\) sufficiently small and \(\ell > 1\), an application of Grönwall's Lemma yields
\[ \sup_{t \in [0,T]} \norm{ \rho_t}_{q} \le \norm{ \rho_0 }_{q}  \big(e^{C T}\big)^{\frac{1}{q}} \le \kappa  \ell  \,\, \text{ a.s.}\]
where 
\[C \doteq C_\epsilon\left(  q(q-1) C_K^2\kappa^2 \ell^2 + q\sqrt{d}4M^2 + qM + 1 \right). \]
This means $\rho^{\xi} \in \mathbb{B}$.
\\
We now show that the map $\mathcal{T}$ is a contraction.

For any $\bar{\xi}, \xi \in \mathbb{B}$, set $\delta \rho=\rho^{\bar{\xi}}-\rho^{\xi}$ and $\delta \xi=\bar{\xi}-\xi$. As before, we apply Itô's formula for the $L^q$-norm of a $H^{1,q}$-valued process in \cite{kry} to $\delta \rho$ and obtain
\begin{align}
&\| \delta \rho_t \|_{q}^q \nonumber\\ &\quad = \frac{1}{2}\int_{\mathbb{T}^d} \int_0^t q | \delta \rho_s (x) |^{q - 2} \delta \rho_s (x)
(\nu_s(x)\nu_s^{\top}(x)+\sigma_s(x)\sigma_s^{\top}(x) )\boldsymbol{\cdot} D^2\delta \rho_s (x) \, ds \, dx\nonumber\\
&\quad - \int_{\mathbb{T}^d} \int_0^t q |  \delta \rho_s (x) |^{q - 2} \delta \rho_s \, [f_{\bar{\xi}}(\rho^{\bar{\xi}},s,x) - f_{\xi}(\rho^{\xi},s,x)]\, ds \, dx\nonumber\\
&\quad - \int_{\mathbb{T}^d} \int_0^t q | \delta \rho_s (x) |^{q - 2} \delta \rho_s (x) \nabla
\delta \rho_s (x) \cdot \sigma_s(x) \, d  B_s \, d  x\nonumber\\
&\quad + \frac{1}{2}  \int_{\mathbb{T}^d} \int_0^t q (q - 1) | \delta \rho_s
(x) |^{q - 2}  | \sigma^{\top}_s(x) \, \nabla \delta \rho_s (x) |^2 \, ds \, dx \nonumber\\
&= \text{(I) + (II) + (III) + (IV)}, \label{ito difference q}
\end{align}

\noindent where the numbered terms are analogous to before with $\delta \rho$ in place of $\rho$, except (II), which has a nonlinear part
\begin{align*}
\text{(II)} &\doteq - \int_{\mathbb{T}^d} \int_0^t q |  \delta \rho_s (x) |^{q - 2} \delta \rho_s \nabla
\cdot [ \rho^{\bar{\xi}}_s (x) F(x,K\ast \bar{\xi}_s)] \, ds \, dx\\
&\quad + \int_{\mathbb{T}^d} \int_0^t q |  \delta \rho_s (x) |^{q - 2} \delta \rho_s \nabla
\cdot [ \rho^{\xi}_s (x) F(x, K\ast \xi_s)] \, ds \, dx\\
&\quad - \int_0^t \int_{\mathbb{T}^d}   q |\delta \rho_s (x) |^{q - 2} \delta \rho_s (x) \Big(2 \sum_{ij} \partial_{j} \delta \rho_s \, \partial_i a^{ij}_s + \delta \rho_s \sum_{ij} \partial_{ij} a^{ij}_s  \Big)  \, dx \, ds
\end{align*}
The above linear part has estimates on $\delta \rho$ following from the same reasoning as before, while for the nonlinear (let us denote it by (II)$^\prime$) we subtract and add $\rho_s^{\xi}(x)$ to get
\begin{align}
\text{(II)}^\prime &= - \int_0^t\int_{\mathbb{T}^d} q |  \delta \rho_s  |^{q - 2} \delta \rho_s \nabla
\cdot [ \delta \rho_s (x) F(x, K\ast \bar{\xi}_s)] \, dx \, ds \nonumber \\
&\quad - \int_0^t \int_{\mathbb{T}^d}  q |  \delta \rho_s|^{q - 2} \delta \rho_s \nabla
\cdot [ \rho^{\xi}_s (x) \, F(x,K \ast  \bar{\xi}_s)] \, dx \, ds \nonumber \\
&\quad + \int_{\mathbb{T}^d} \int_0^t q |  \delta \rho_s|^{q - 2} \delta \rho_s \nabla
\cdot [ \rho^{\xi}_s (x) F(x, K\ast \xi_s )] \, ds \, dx \nonumber 
\end{align}
So integration by parts gives
\begin{align}
\text{(II)}^\prime &\le q(q-1) \int_0^t  \int_{\mathbb{T}^d} |\delta \rho_s  |^{q - 2} \Big( |\nabla \delta \rho_s| |\delta \rho_s | L + |\nabla \delta \rho_s| \,|\rho_s^\xi|\, L\,|K \ast \delta \xi_s| \Big) \, dx \, ds \nonumber \\
&\le q(q-1)L \int_0^t  \int_{\mathbb{T}^d} |\delta \rho_s  |^{q - 2} \Big( |\nabla \delta \rho_s| |\delta \rho_s |  + |\nabla \delta \rho_s| \,|\rho_s^\xi|\, \|\delta \xi_s\|_{q} \Big) \, dx \, ds \label{II_contraction}
\end{align}
\noindent Above, we also used $(\mathbf{A}^K)$ and that $F$ is bounded and Lipschitz $(\mathbf{A}^F)$. In fact, one can show similar bounds assuming $(\mathbf{A}^K)$ and  $(\mathbf{A}^I)$, noticing that $F(x,u)=u$ is also Lipschitz, and the boundedness of $F$ being replaced by that of $K\ast \bar{\xi}_s$, since $\|K \ast \bar{\xi}_s\|_{\infty} \le C_K \|\bar{\xi}_s\|_{q}\leq C_K \kappa \ell$.

Next, by $\epsilon$-Young inequality, we have
\[|\delta \rho_s  |^{q - 1} |\nabla \delta \rho_s| = |\delta \rho_s  |^{q - 2}|\delta \rho_s| |\nabla \delta \rho_s| \le C_\epsilon |\delta  \rho_s  |^{q} + \epsilon |\delta \rho_s  |^{q-2}|\nabla \delta \rho_s|^2 \]

and also
\begin{align*}
&\big \||\delta \rho_s  |^{q - 2} |\nabla \delta \rho_s| \, \rho_s^\xi \, \|\delta \xi_s\|_q \, \big \|_1 \\ 
&\le \big \||\delta \rho_s  |^{\frac{q - 2}{2}} |\nabla \delta \rho_s| \big \|_2 \, \big \||\delta \rho_s  |^{\frac{q - 2}{2}} \rho_s^\xi \, \|\delta \xi_s\|_q\, \big \|_2 \\
&\le \epsilon \, \big \||\delta \rho_s  |^{\frac{q - 2}{2}} |\nabla \delta \rho_s| \big\|^2_2 \, + \, C_\epsilon \, \big \||\delta \rho_s  |^{\frac{q - 2}{2}} \rho_s^\xi \, \|\delta \xi_s\|_q \big \|^2_2
\end{align*}

Again, by Young's inequality, we have
\begin{align*}
\big\||\delta \rho_s  |^{\frac{q - 2}{2}} \rho_s^\xi \, \|\delta \xi_s\|_q \, \big\|^2_2 &= \big \| |\delta \rho_s  |^{q - 2} (\rho_s^\xi)^2 \, \|\delta \xi_s\|_q^2 \, \big \|_1 \\
&\le \, \big \|\delta \rho_s \big \|_q^{q} \, + \,  \, \big \|\rho_s^\xi \, \|\delta \xi_s\|_q  \big \|_q^q \\
&= \, \|\delta \rho_s \|_q^{q} \, + \,  \, \|\rho_s^\xi  \|_q^q \, \|\delta \xi_s\|^q_{q}
\end{align*}

Using the above in (\ref{II_contraction}), we get
\begin{align*}
\text{(II)}^\prime &\le 2C_\epsilon q(q-1)L \int_0^t \|\delta \rho_s\|_{q}^q \, ds + 2\epsilon \, q (q-1) L \int_0^t \int_{\Rd} |\delta \rho_s|^{q-2}|\nabla \delta \rho_s|^2 \, dx \, ds \\
& \quad + C_\epsilon q(q-1)L \int_0^t  \|\rho_s^\xi \|^q_{q} \, \|\delta \xi_s\|_{q} ^q \, ds\\
&\le 2C_\epsilon q(q-1)L \int_0^t \|\delta \rho_s\|_{q}^q \, ds + 2\epsilon \, q (q-1)L \int_0^t \int_{\Rd} |\delta \rho_s|^{q-2}|\nabla \delta \rho_s|^2 \, dx \, ds \\
& \quad + C_\epsilon q(q-1)L \,  \, \Big\|\|\rho^{\xi}\|_{T,q}\Big\|_{\infty}^q  \int_0^t   \|\delta \xi_s\|^q_{q} \, ds 
\end{align*}

\medskip
Finally, similarly to before, by choosing $\epsilon >0$
sufficiently small, we obtain 
\begin{align*}
\| \delta \rho_t \|_{q}^q &=  \text{(I) + (II) + (III) + (IV)} \\
&\le C \int_0^t \|\delta \rho_s\|_{q}^{q}\, ds + C \, (\kappa \ell)^q \int_0^t \|\delta \xi_s\|^q_{q} \, ds 
\end{align*}

Thus, an application of Gr\"onwall's lemma yields
\[ \Big\|\|\delta \rho\|_{T,q}\Big\|_{\infty} \le  e^{\frac{CT}{q}} (C)^{\frac{1}{q}} (\kappa \ell) T^\frac{1}{q}\Big\|\|\delta \xi\|_{T,q}\Big\|_{\infty}  
\] 
and taking $T$ small, depending on all given data $\kappa,\lambda, C_K,q,L,M$ and $d$, we obtain that $\mathcal{T}$ is a contraction. 

\bigskip

\end{proof}

\section*{Acknowledgements}

C. Olivera  is partially supported by  FAPESP-ANR by the  grant  $2022/03379-0$ ,  by FAPESP by the grant  $2020/04426-6$, and  CNPq by the grant $422145/2023-8$. A. B. de Souza is partially supported by  Coordenação de Aperfeiçoamento de Pessoal de Nível Superior – Brasil (CAPES) – Finance Code $001$. J. Knorst  is partially supported by  FAPESP by the  grants  2022/13413-0,\\  BEPE 2023/14629-0.

\end{document}